\documentclass[a4paper,11pt]{amsart}

\usepackage[latin1]{inputenc}
\usepackage[cyr]{aeguill}
\usepackage[english]{babel}

\usepackage[T1]{fontenc}
\usepackage{amsmath}
\usepackage{amsfonts}
\usepackage{amssymb}
\usepackage{amsthm,eepic}
\usepackage{mathrsfs}
\usepackage{enumitem}
\usepackage{graphicx,tikz}
\usepackage{footnote}
\usepackage{hyperref}
\hypersetup{colorlinks=true,    linkcolor=blue,
citecolor=red,       filecolor=BrickRed,       urlcolor=darkgreen
}

\renewcommand{\geq}{\geqslant}
\renewcommand{\leq}{\leqslant}
\newtheorem{thm}{Theorem}

\newtheorem{rem}[thm]{Remark}
\newtheorem{cor}[thm]{Corollary}
\newtheorem{prop}[thm]{Proposition}
\newtheorem{lem}[thm]{Lemma}

\newtheorem*{prop*}{Proposition}
\renewcommand{\epsilon}{\varepsilon}

\usepackage{color}
\definecolor{darkgreen}{rgb}{0,0.4,0}
\definecolor{MyDarkBlue}{rgb}{0,0.08,0.50}
\definecolor{BrickRed}{rgb}{0.65,0.08,0}

\title{Harmonic functions for singular quadrant walks}

\author{Viet Hung Hoang} 
\address{Institut Denis Poisson, UMR CNRS 7013, Universit\'e de Tours et Universit\'e d'Orl\'eans, Parc de Grandmont, 37200 Tours, France}
\email{viet.hung-hoang@lmpt.univ-tours.fr}
\author{Kilian Raschel}
\address{Laboratoire Angevin de Recherche en Math\'ematiques, UMR CNRS 6093, Universit\'e d'Angers, 2 Boulevard Lavoisier, 49000 Angers, France}
\email{raschel@math.cnrs.fr}
\author{Pierre Tarrago}

\address{Laboratoire de Probabilit\'es, Statistique et Mod\'elisation, UMR CNRS 8001, Sorbonne Universit\'e, 4 Place Jussieu, 75005 Paris, France}
\email{pierre.tarrago@sorbonne-universite.fr}
\thanks{This project has received funding from the European Research Council (ERC) under the European Union's Horizon 2020 research and innovation programme under the Grant Agreement No.\ 759702, from Centre Henri Lebesgue (programme ANR-11-LABX-0020-0), from the ANR DeRerumNatura (ANR-19-CE40-0018) and ANR CORTIPOM (ANR-21-CE40-0019).}


\date{\today}

\begin{document}

\begin{abstract}
We consider discrete (time and space) random walks confined to the quarter plane, with jumps only in directions $(i,j)$ with $i+j\geq0$ and small negative jumps, i.e., $i,j\geq -1$. These walks are called singular, and were recently intensively studied from a combinatorial point of view. In this paper, we show how the compensation approach introduced in the 90ies by Adan, Wessels and Zijm may be applied to compute  positive harmonic functions with Dirichlet boundary conditions. In particular, in case the random walks have a drift with positive coordinates, we derive an explicit formula for the escape probability, which is the probability to tend to infinity without reaching the boundary axes. These formulas typically involve famous recurrent sequences, such as the Fibonacci numbers. As a second step, we propose a probabilistic interpretation of the previously constructed harmonic functions and prove that they allow to compute all positive harmonic functions of these singular walks. To that purpose, we derive the asymptotics of the Green functions in all directions of the quarter plane and use Martin boundary theory.
\end{abstract}

\maketitle

\section{Introduction}
\subsection*{Espace probability for random walks in cones}
Consider a multidimensional lattice random walk $\{S(n)\}_{n\geq 1}$, i.e., $S(n) = X(1)+\cdots + X(n)$ for all $n\geq 1$. Given a cone $C\subset \mathbb R^d$, $d\geq 1$, introduce the associated first exit time
\begin{equation*}
   \tau_x = \inf\{n\geq 0 : x+S(n)\notin C\}\leq \infty.
\end{equation*}
If the drift of the increment distribution belongs to the cone, then (ignoring pathological behaviours) for $x$ interior to $C$, the escape probability (also called survival probability) 
\begin{equation}
\label{eq:survival_probability}
   \mathbb P(\tau_x=\infty) = \mathbb P(\forall n\geq 0, x+S(n)\in C)
\end{equation}
is strictly positive, and defines a discrete harmonic function with Dirichlet boundary conditions.

The question at the origin of the present work is the following: does this natural harmonic function admit an expression in closed form? The harmonicity property is equivalent to a recurrence relation, which in dimension $1$ may be easily solved, at least in the bounded jump case. On the other hand, it is known that the behaviour of solutions to multivariate recurrences is much harder and vast \cite{BMPe-00,BMPe-03}, a fortiori with boundary conditions depending on a cone; there is no hope, in general, to compute explicitly the escape probability \eqref{eq:survival_probability}.

\subsection*{A glimpse of our results (Part~\ref{part:construction})}

We may now state the contributions of our paper, which consists of two parts. We will introduce a class of singular random walks in dimension $2$ (see the next subsection for a precise definition, see also Figure~\ref{fig:step_sets}), and look at the case of the cone $C$ being the positive quarter plane. In Part~\ref{part:construction}, we will produce explicit expressions for the escape probability \eqref{eq:survival_probability}. To that purpose, we will use the compensation approach, as introduced in \cite{AdWeZi-90,Ad-91,AdWeZi-93} by Adan, Wessels and Zijm. Before giving more details both on the random walks considered and on the techniques used, let us present an explicit example.

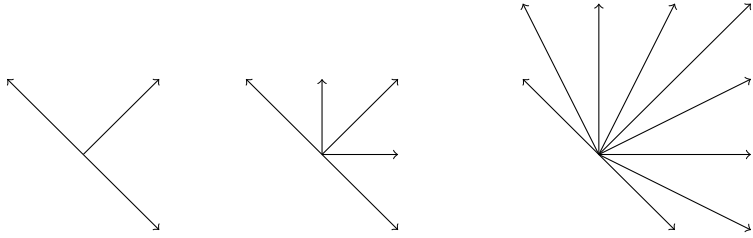
\begin{figure}
\begin{center}
\begin{tikzpicture}[scale=1.0] 
    \draw[->,white] (1,2) -- (-1,-2);
    \draw[->,white] (1,-2) -- (-1,1);
    \draw[->] (0,0) -- (1,-1);
    \draw[->] (0,0) -- (-1,1);
    \draw[->] (0,0) -- (1,1);
  \end{tikzpicture}
\begin{tikzpicture}[scale=1.0] 
    \draw[->,white] (2,2) -- (-2,-2);
    \draw[->,white] (2,-2) -- (-2,2);
    \draw[->] (0,0) -- (0,1);
    \draw[->] (0,0) -- (1,1);
    \draw[->] (0,0) -- (1,-1);
    \draw[->] (0,0) -- (-1,1);
    \draw[->] (0,0) -- (1,0);
  \end{tikzpicture}
     \begin{tikzpicture}[scale=1.0] 
  \draw[->,white] (2.5,2) -- (-1.5,-2);
    \draw[->,white] (2,-2) -- (-1.5,2);
    \draw[->] (0,0) -- (2,1);
    \draw[->] (0,0) -- (2,2);
    \draw[->] (0,0) -- (2,0);
    \draw[->] (0,0) -- (1,2);
    \draw[->] (0,0) -- (0,2);
    \draw[->] (0,0) -- (-1,2);
    \draw[->] (0,0) -- (2,-1);
    \draw[->] (0,0) -- (1,-1);
    \draw[->] (0,0) -- (-1,1);
  \end{tikzpicture}
  \end{center}
  \vspace{-5mm}
  \caption{Various step sets. From left to right: the simplest singular random walk; an arbitrary small step singular random walk; an example of singular walk with bigger jumps.}
  \label{fig:step_sets}
\end{figure}

Consider in this paragraph the model $p_{-1,1}=p_{1,1}=p_{1,-1}=\frac{1}{3}$, as on Figure~\ref{fig:step_sets} (left). To that example, our main result will entail that the escape probability starting at $x=(i,j)$ equals 
   \begin{equation}
   \label{eq:expression_escape_Fibonacci}
      \mathbb P(\tau_{(i,j)}=\infty) = 1-\frac{1}{2^i}-\frac{1}{2^j}+\frac{1}{2^i5^j}+\frac{1}{2^j5^i}-\frac{1}{5^i13^j}-\frac{1}{5^j13^i}+\frac{1}{13^i34^j}+\frac{1}{13^j34^i}+\cdots.
   \end{equation}
The integers appearing in the denominators in \eqref{eq:expression_escape_Fibonacci}, namely
$
    1, 2, 5, 13, 34, 89, 233, 610,\ldots
$
are directly related to Fibonacci sequence (\href{https://oeis.org/A000045}{A000045}), as the $n$-th term in the sequence is equal to $F_{2n-1}$. For $i=j=1$, this result is derived in \cite[Prop.~9]{MiRe-09}, using a functional equation approach.

In this paper, we will prove similar formulas for other singular random walks; in particular, we will see how to define the sequence of denominators in general (providing a nice interplay between probabilistic and arithmetic properties of these singular random walks). As a second step, we will prove that similar expressions hold for infinitely many positive harmonic functions, of the form \begin{equation}
\label{eq:f(i,j)_Ansatz}
    h(i,j) = \sum_{n\geq 0} c_n \alpha_n^i\beta_n^j,
\end{equation}
where the $\alpha_n,\beta_n,c_n$ are real constants (to be specified).

\subsection*{A glimpse of our results (Part~\ref{part:Martin_boundary})}

The main objective of our second part is to propose a probabilistic interpretation of the harmonic functions \eqref{eq:f(i,j)_Ansatz} constructed via the compensation approach. Our central result is to prove that the previous harmonic functions actually allow to construct all positive harmonic functions. In other words, the compensation approach yields an exhaustive description of positive harmonic functions for singular random walks. In concrete terms, this means that there is a correspondence between minimal positive harmonic functions and the yellow domain on the left display on Figure~\ref{fig:MB_and_escargot}, exactly the same phenomenon as in the non-singular case \cite{IRLo-10}. This result is properly stated in Corollary \ref{cor:MB}.

Among all positive harmonic functions, there is a unique bounded positive harmonic function, which is the escape probability. It corresponds asymptotically by looking at the Green function 
\begin{equation}
\label{eq:def_G}
    G(x,y) = \sum_{n\geq 1}\mathbb P(x+S(n) = y,\tau_x>n)
\end{equation}
along the drift direction.

To prove these results, we use Martin boundary theory, and we believe that several intermediate results are of independent interest. The key idea is to compute the asymptotics of the Green function \eqref{eq:def_G}
as $x$ is fixed and $y$ goes to infinity in any direction of the cone. See Theorem~\ref{thm:asymptotic_inside_green} for the main statement. While such results have been recently derived in a close context (zero drift random walks \cite{DuRaTaWa-22}, irreducible non-zero drift random walks \cite{IRLo-10}), there was no version in the literature applying to our context. The case of a boundary direction needs a particular attention, due to the interaction with the axes. 

From a technical point of view, Part~\ref{part:Martin_boundary} is inspired by the work \cite{IRLo-10}, \cite{DeWa-15} and \cite{DuRaTaWa-22}.

\subsection*{Singular random walks in the quadrant}
Throughout the paper, we assume that the distribution of the increments $X(i)$ of the random walk have transition probabilities $p_{i,j}$ in $\mathbb Z^2$ such that (see Figure~\ref{fig:step_sets}):
\begin{enumerate}
	\item\label{it:norm}$\sum_{i,j} p_{i,j} = 1$ \textit{(normalization)};
	\item\label{it:small_neg}$p_{i,j}=0$ for all $i\leq -2$ or $j\leq -2$ \textit{(small negative jumps)};
	\item $p_{-1,-1}=p_{-1,0}=p_{0,-1}=0$ \textit{(singular walks)};
	\item $p_{-1,1}p_{1,-1}\not= 0$;
	\item There exists $(i,j)$ with $i+j>0$ such that $p_{i,j}>0$ \textit{(non-degeneracy)};
	\item\label{it:Laplace}The $p_{i,j}$ admit exponential moments in the following sense: the Laplace transform $\sum_{i,j} p_{i,j}e^{ix+jy}$ is finite in a neighborhood of any point of the curve \eqref{eq:Ney-Spitzer_curve} \textit{(moment assumption)}.
\end{enumerate}
By definition, the kernel of the model is the bivariate polynomial
\begin{equation}
\label{eq:kernel}
   K(\alpha,\beta) = \alpha\beta\left(\sum_{i,j\geq -1} p_{i,j} \alpha^i\beta^j -1\right).
\end{equation}

\subsection*{Compensation approach}
This technique has been developed in the probabilistic context of stationary distributions for random walks, see \cite{AdWeZi-90,Ad-91,AdWeZi-93}. It does not aim directly at obtaining a solution for a generating function (as it is usual for quadrant walk problems), but rather tries to find a solution for its coefficients, in our case the escape probability $h(i,j) = \mathbb P(\tau_{(i,j)}=\infty)$ starting from $(i,j)$. 

As a discrete harmonic function, $h(i,j)$ satisfies certain recursion relations (coming from harmonicity in our case), which differ depending on whether the state $(i, j)$ lies on the boundary or not:
\begin{align}
\label{eq:harmonicity_relation_interior}
   h(i,j) &= \sum_{k,\ell\geq -1} p_{k,\ell} h(i+k,j+\ell),\qquad \forall i,j\geq 1,\\
   h(i,0)&=0,\qquad \forall i\geq 0,\label{eq:Dirichlet_H}\\
   h(0,j)&=0,\qquad \forall j\geq 0.\label{eq:Dirichlet_V}
\end{align}
The idea is then to express $h(i,j)$ as a linear combination of products $\alpha^i\beta^j$, for pairs $(\alpha,\beta)$ such that the recursion relations \eqref{eq:harmonicity_relation_interior} in the interior of the quarter plane hold. This is equivalent to choosing the parameters such that $K(\alpha,\beta)=0$, with $K$ as in \eqref{eq:kernel}. The products have to be chosen such that the recursion relations on the boundaries \eqref{eq:Dirichlet_H} and \eqref{eq:Dirichlet_V} are satisfied as well. As it turns out, this can be done by alternatingly compensating for the errors on the two boundaries, which eventually leads to an infinite series of product forms. The typical outcome of the compensation approach is an expression of the form \eqref{eq:f(i,j)_Ansatz}. This clearly formally contains the series presented in \eqref{eq:expression_escape_Fibonacci}.

As a side note, our work proposes a new example of applicability of the compensation approach, in relation with potential theory and discrete harmonic functions.

\subsection*{Related literature on explicit formulas for escape probabilities}
In the quarter plane, the few known formulas concern non-singular walks with certain finite reflection groups (for which an algebraic version of the reflection principle applies); see Figure~\ref{fig:step_sets_non-sing} for three examples. For instance, for the simple random walk (leftmost display on the figure), the escape probability equals \cite[Cor.~8]{KuRa-11}
\begin{equation*}
    \mathbb P(\tau_{(i,j)} = \infty) = \left(1-\Bigl(\frac{p_{-1,0}}{p_{1,0}}\Bigr)^i\right)\left(1 -\Bigl(\frac{p_{0,-1}}{p_{0,1}}\Bigr)^j\right),
\end{equation*}
which interestingly corresponds to a finite (four terms) sum in \eqref{eq:f(i,j)_Ansatz}. Similar expressions hold for the other two models on Figure~\ref{fig:step_sets_non-sing}, with a sum involving six and eight terms, respectively.
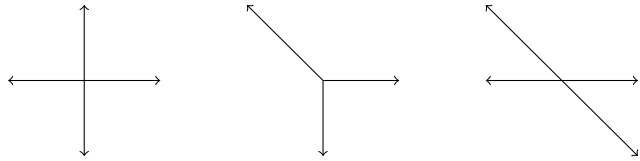
\begin{figure}
\begin{center}
\begin{tikzpicture}[scale=1.0] 
    \draw[->,white] (1,2) -- (-1,-2);
    \draw[->,white] (1,-2) -- (-1,1);
    \draw[->] (0,0) -- (1,0);
    \draw[->] (0,0) -- (-1,0);
    \draw[->] (0,0) -- (0,1);
    \draw[->] (0,0) -- (0,-1);
  \end{tikzpicture}
\begin{tikzpicture}[scale=1.0] 
    \draw[->,white] (2,2) -- (-2,-2);
    \draw[->,white] (2,-2) -- (-2,2);
    \draw[->] (0,0) -- (-1,1);
    \draw[->] (0,0) -- (1,0);
    \draw[->] (0,0) -- (0,-1);
  \end{tikzpicture}
     \begin{tikzpicture}[scale=1.0] 
    \draw[->,white] (1,2) -- (-1,-2);
    \draw[->,white] (1,-2) -- (-1,1);
    \draw[->] (0,0) -- (1,0);
    \draw[->] (0,0) -- (-1,0);
    \draw[->] (0,0) -- (-1,1);
    \draw[->] (0,0) -- (1,-1);
  \end{tikzpicture}
  \end{center}
  \vspace{-5mm}
  \caption{Three non-singular step sets with finite reflection groups. The rightmost example is assumed to satisfy $p_{-1,1}p_{1,-1}=p_{-1,0}p_{1,0}$.}
    \label{fig:step_sets_non-sing}
\end{figure}

Still in the non-singular case, other approaches based on complex analysis techniques (boundary value problems) allow to compute the generating functions
\begin{equation*}
    \sum_{i,j\geq 1} h(i,j) x^{i}y^{j}
\end{equation*}
of harmonic functions in terms of certain conformal mappings, see \cite{HoRaTa-22} (zero drift case) and \cite{LeRa-16} (non-zero drift case). 

In the continuous setting (Brownian motion in cones), there is a unified formula \cite[Thm~C]{GaRa-14} for the probability of escape from arbitrary cones $C\subset\mathbb R^d$, $d\geq 1$. More precisely, if the drift $\vec d$ is interior to the cone, then
\begin{equation*}
    \mathbb P(\tau_x = \infty) = (2\pi)^{d/2} \exp\bigl(\bigl\vert x-\vec d\bigr\vert^2\bigr) p^C\bigl(1,x,\vec d\bigr),
\end{equation*}
where $p^C$ denotes the heat kernel of the cone. The question of finding explicit expressions is then reduced to computing in closed-form the heat kernel, which is more classical. For example, if the cone is a Weyl chamber of type $A$, the determinantal Karlin-McGregor formula holds, and the survival probability is a finite sum of product forms, as in \eqref{eq:f(i,j)_Ansatz}. 

\subsection*{Applications and refinements of our results}
We conclude the introduction by mentioning related open questions and potential applications of our results. 
\begin{itemize}
    \item One advantage of the compensation approach is that it is not based on generating functions. However, it would be interesting to understand how the compensation approach is related to Cohen and Boxma techniques \cite{CoBo-83,Co-92}, which we developed in the paper \cite{HoRaTa-22} to study harmonic functions for non-singular random walks.\smallskip
   \item In principle, the compensation approach could be applied as well to harmonic functions for reflected random walks in the quarter plane (with Neumann boundary conditions). It should also provide explicit expressions for $t$-harmonic functions (Dirichlet or Neumann boundary conditions).\smallskip
   \item Formulas of the type of \eqref{eq:f(i,j)_Ansatz} hold for singular random walks (our main result) and for a few other (finite group) models, as those represented on Figure~\ref{fig:step_sets_non-sing}. This clearly suggests the question of determining the class of models for which harmonic functions may be expressed as in \eqref{eq:f(i,j)_Ansatz}, via the pattern:\smallskip
   \item[]
       $\text{model} \to \text{curve} \to \text{sequence of numbers} \to \text{harmonic function (via Ansatz \eqref{eq:f(i,j)_Ansatz})}.$ \smallskip
   \item Is it possible to study more generally the Martin boundary and the Green functions asymptotics of singular random walks in arbitrary cones? What is the exact applicability of the compensation approach? For instance, can it be applied to harmonic functions of higher dimensional models?
   \item The compensation formula \eqref{eq: h(i,j)} that we find in the context of harmonic functions for singular random walks in the quarter plane is reminiscent of the alternating formula for the harmonic function of a space-time Brownian motion $(t,B_t)_{t\geq 0}$ conditioned to stay in the cone $\{(t,y),0< y<t\}$, see \cite{Def-15}. This suggests a relation between the singular walks we are considering and stochastic processes conditioned to stay in affine Weyl chambers.
\end{itemize}

\part{Constructing harmonic functions via the compensation approach}
\label{part:construction}

\section{Properties of the Laplace transform and convergence of the series}

This section aims at constructing harmonic functions as series taking the form of \eqref{eq:f(i,j)_Ansatz}, by using the compensation approach. We first introduce a curve $\mathcal{K}\in\mathbb{R}^2$ such that $\alpha^i\beta^j$ is a solution of \eqref{eq:harmonicity_relation_interior} for any $(\alpha,\beta)\in\mathcal{K}$, see Section~\ref{subsec:level_sets}. In Section~\ref{subsec:construction}, we then construct sequences $(\alpha_n,\beta_n)_{n\geq0}\subset\mathcal{K}$ such that their series satisfies the boundary conditions \eqref{eq:Dirichlet_H}--\eqref{eq:Dirichlet_V}. We then prove that these series converge and do not depend on the initial starting point of the sequence, see Section~\ref{subsec:starting}. Finally, in Section~\ref{subsec:boundary}, we study a particular case, where we should renormalize the harmonic function to get a non-zero quantity.

\subsection{Level sets of the Laplace transform}
\label{subsec:level_sets}

We first observe that the product form $\alpha^i\beta^j$ is solution to \eqref{eq:harmonicity_relation_interior} if and only if
\begin{equation}
    \alpha\beta = \sum_{k,\ell} p_{k,\ell}\alpha^{k+1}\beta^{\ell+1},
\end{equation}
or equivalently, using the kernel notation \eqref{eq:kernel}, if and only if $K(\alpha,\beta)=0$. We therefore introduce the algebraic curve
\begin{equation}
\label{eq:curve_K}
    \mathcal{K}:=\{(\alpha,\beta)\in\mathbb{R}_{\geq 0}^2:K(\alpha,\beta)=0\}.
\end{equation}
See Figure~\ref{fig:mathcal_K_non-convex} for an example.

To take advantage of some convexity properties, we mainly investigate $\mathcal{K}$ through an alternative exponential scaling as follows:
\begin{equation}
\label{eq:Ney-Spitzer_curve}
	\mathcal{G}:=\{ (x,y)\in\mathbb{R}^2: K(e^x,e^y)=0 \}.
\end{equation}
See Figures~\ref{fig:mathcal_K_non-convex} and \ref{fig:MB_and_escargot} for examples of curves $\mathcal G$. Let us further denote
\begin{equation}
\label{eq:def_G+-}
	\mathcal{G}^+:= \{ (x,y)\in\mathbb{R}^2: K(e^x,e^y)<0 \} \quad\text{and}\quad
	\mathcal{G}^-:= \{ (x,y)\in\mathbb{R}^2: K(e^x,e^y)>0 \}.
\end{equation}
The following lemma presents some crucial properties of the curve $\mathcal{G}$.
\begin{lem}\label{lem: properties of mathcal G}
	Under Assumptions~\ref{it:norm}--\ref{it:Laplace}, we have the following assertions:
	\begin{enumerate}
		\item\label{item: mathcal G^+ convex} $\mathcal{G}^+$ is an unbounded convex domain and includes the ray $\{t(1,1):t<0\}$;
		\item\label{item:(0,0)_in_mathcal_G_and_its_tails} $\mathcal{G}$ passes through $(0,0)$ and the two tails of $\mathcal{G}$ lie in the third quadrant $\mathbb{R}_-^2$;
		\item\label{item: tangent at (0,0)} $\mathcal{G}$ admits a tangent at $(0,0)$ satisfying the equation
		\begin{equation*}
			\left(\sum_{k,\ell}kp_{k,\ell}\right)x + 	\left(\sum_{k,\ell}\ell p_{k,\ell}\right)y = 0.
		\end{equation*}
	\end{enumerate}
\end{lem}

\begin{proof}
	We first prove~\ref{item: mathcal G^+ convex}. $K(e^x,e^y)$ can be factorized as
	\begin{equation*}
		K(e^x,e^y) = e^{x+y}G(x,y),
	\end{equation*}
	where $G(x,y)$ is defined as
	\begin{equation}\label{eq: def G(x,y)}
	    G(x,y):= \sum_{k,\ell} p_{k,\ell} e^{kx+\ell y}-1.
	\end{equation}
	Since $\sum_{k,\ell} p_{k,\ell} e^{kx+\ell y}$ is a moment-generating function of a random variable in $\mathbb{Z}^2$, it is convex in $\mathbb{R}^2$. The domain $\mathcal{G}^+$ is thus convex. By letting $x=y<0$, we then have:
	\begin{equation*}
		G(x,x) =  \sum_{k,\ell} p_{k,\ell} e^{(k+\ell)x }-1 <\sum_{k,\ell} p_{k,\ell}  -1 =0,
	\end{equation*}
	since $p_{k,\ell}=0$ if $k+\ell<0$. Thus, $\mathcal{G}^+$ is unbounded and includes the ray $\{t(1,1):t<0\}$.
	
	We now prove~\ref{item:(0,0)_in_mathcal_G_and_its_tails}. Since $G(0,0)=0$, then $(0,0)\in\mathcal{G}$. By letting $x=0>y$ with $\vert y\vert$ large enough, we have
	\begin{equation*}
		G(0,y)=  \sum_{k,\ell} p_{k,\ell} e^{\ell y}-1 > p_{1,-1}e^{-y} >0.
	\end{equation*}
	Similarly, for all $y=0>x$ with $\vert x\vert $ large enough, we have
	\begin{equation*}
		G(x,0)=  \sum_{k,\ell} p_{k,\ell} e^{kx}-1 > p_{-1,1}e^{-x} >0.
	\end{equation*}
	It implies that both tails of $\mathcal{G}$ lie in $\mathbb{R}_-^2$.
	
	Finally, \ref{item: tangent at (0,0)} is easily seen from the equation of the tangent at $(0,0)$, which is
	\begin{equation*}
		\partial_xK(e^0,e^0)x + \partial_yK(e^0,e^0)y = 0.\qedhere
	\end{equation*}
\end{proof}

Although $\mathcal{G}^+$ is always convex, the bounded domain delimited by $\mathcal K$ is not necessarily convex (see Figure~\ref{fig:mathcal_K_non-convex}). This is the main reason for us to use the curve $\mathcal{G}$ rather than $\mathcal{K}$. 

\begin{figure}[t]
    \centering
    \includegraphics[scale=0.25]{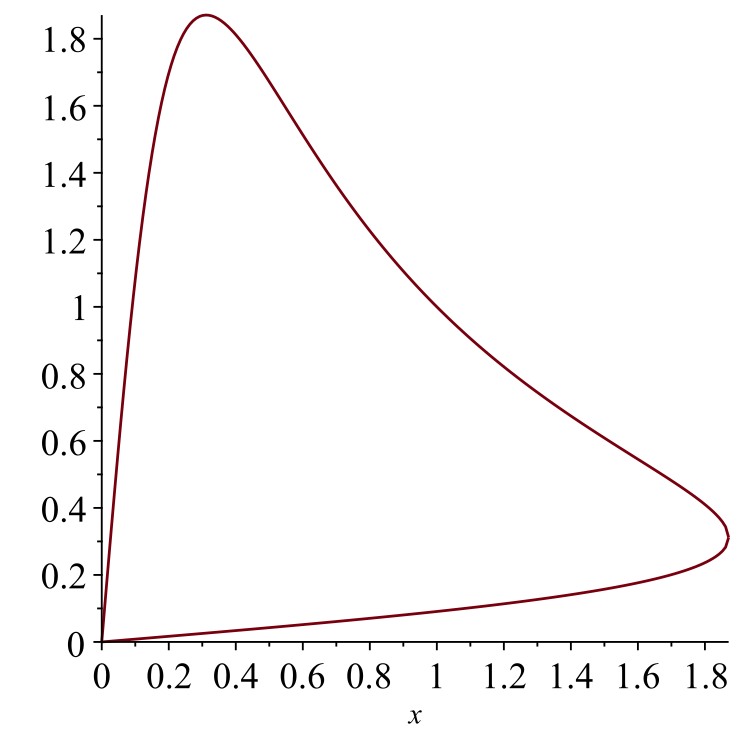}
    \includegraphics[scale=0.25]{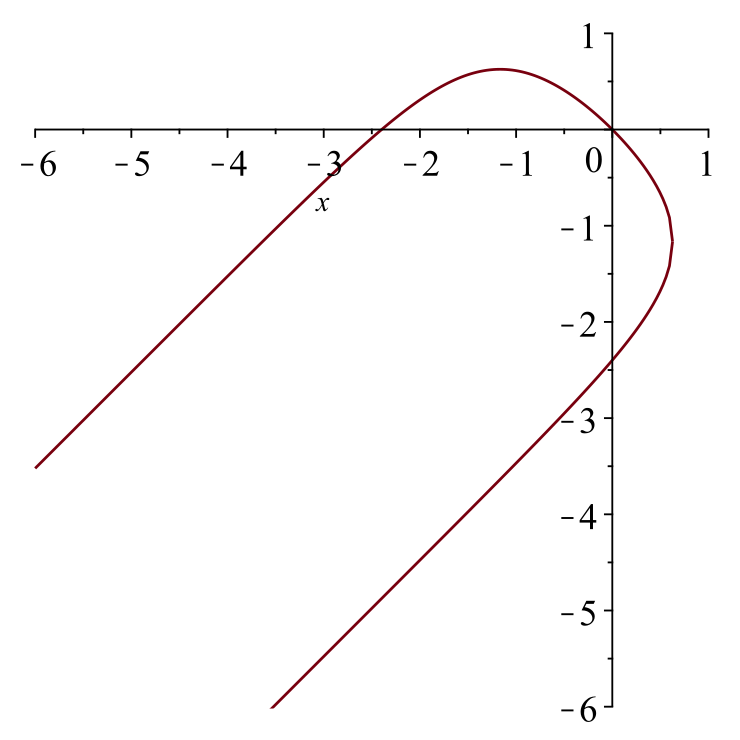}
    \caption{The model $p_{1,1}=5/6$, $p_{1,-1}=p_{-1,1}=1/12$: the interior of $\mathcal{K}$ (left) is not convex, while the interior of $\mathcal{G}$ (right) is convex.}
    \label{fig:mathcal_K_non-convex}
\end{figure}

We now parametrize the curve $\mathcal{G}$ by functions. Define $f:\mathbb{R}_{\leq 0}\to\mathbb{R}$ such that for all $x\leq0$, $K(e^x,e^{f(x)})=0$ and $f(x)\geq x$. Similarly, we define $g:\mathbb{R}_{\leq 0}\to\mathbb{R}$ such that $K(e^{g(y)},e^y)=0$ and $g(y)\geq y$ for all $y\leq 0$. The following lemma describes key properties of $f$ and $g$.

\begin{lem}
\label{lem:definition_fg}
	Under Assumptions~\ref{it:norm}--\ref{it:Laplace}, we have the following assertions:
	\begin{enumerate}
		\item\label{item: f, g well defined, concave} $f$ and $g$ are well defined, concave, and infinitely differentiable on $\mathbb{R}_-$;
		\item\label{item: max f(x) and f(-infty)} $\max_{x\leq 0}f(x)>f(0)=0$ and $\arg \max_{x\leq 0} f(x)$ includes a unique point, denoted by $x_0$. Further, $f(x)$ is strictly increasing on $(-\infty,x_0)$, strictly decreasing on $(x_0,0)$, and $\lim_{x\to-\infty}f(x) = -\infty$;
		\item\label{item: max g(y) and g(-infty)} $\max_{y\leq 0}g(y)> g(0)=0$ and $\arg \max_{y\leq 0} g(y)$ includes a unique point, denoted by $y_0$. Further, $g(y)$ is strictly increasing on $(-\infty,y_0)$, strictly decreasing on $(y_0,0)$, and $\lim_{y\to-\infty} g(y) = -\infty$.
	\end{enumerate}
\end{lem}

\begin{proof}
    We first prove Item~\ref{item: f, g well defined, concave}. By the convexity of $\mathcal{G}^+$, the slope of $\mathcal{G}$ at $(0,0)$, and the tails' position of $\mathcal{G}$ (see Lemma~\ref{lem: properties of mathcal G}), it is easily seen that $f$ and $g$ are well defined and concave on $\mathbb{R}_{\leq 0}$. Since any point $(x,y)\in\mathcal{G}$ is not the minimiser of the convex function $G$ defined in \eqref{eq: def G(x,y)}, then $\partial_xG$ and $\partial_yG$ cannot vanish simultaneously for any $x,y\in\mathcal{G}$. By the construction of $f$ and $g$, we then have
    \begin{equation*}
        \partial_x G(x,f(x)) \not= 0 \quad\text{and}\quad
        \partial_y G(g(y),y) \not= 0
    \end{equation*}
    for all $x,y\leq 0$. Thus, 
    \begin{equation*}
        f' (x) = -\frac{\partial_xG(x,f(x))}{\partial_yG(x,f(x))}\quad  \text{and}\quad g' (y) = -\frac{\partial_yG(g(y),y)}{\partial_xG(g(y),y)}
    \end{equation*}
    are well defined for all $x,y<0$, and so are all the higher orders derivatives of $f$ and $g$.
	
	We now prove Item~\ref{item: max f(x) and f(-infty)}. Since
	\begin{equation*}
		f'(0) = -\frac{\partial_xK}{\partial_yK} (e^0,e^0) = \-\frac{\sum_{k,\ell}kp_{k,\ell}}{\sum_{k,\ell}\ell p_{k,\ell}} <0,
	\end{equation*}
	then $\max_{x\leq 0}f(x)>f(0)=0$. If $\arg \max_{x\leq 0} f(x)$ includes at least two distinct points, then it also contains the segment between these points by the concavity of $f$. This implies that $K(\exp(x),\exp(\max_{x\leq 0}f(x)))$ is a constant function of $x$, which does not hold true. Hence, $\arg \max_{x\leq 0} f(x)$ includes a unique point $x_0$. The strict monotonicity of $f$ on $(\infty,x_0)$ and $(x_0,0)$ then follows. By letting $x_1<x_0$ such that $f(x_1)=0$, we have for all $x<x_1$,
	\begin{equation*}
		f(x) < \frac{f(x_0)}{x_0-x_1}(x-x_1)\overset{x\to -\infty}{\longrightarrow}-\infty.
	\end{equation*}
	Thus, $f(x)\to -\infty$ as $x\to - \infty$.
	
	Item~\ref{item: max g(y) and g(-infty)} is proven similarly to \ref{item: max f(x) and f(-infty)}. The proof is then complete.
\end{proof}

\subsection{Construction of the product forms}
\label{subsec:construction}

Our main objective here is to prove Lemma~\ref{lem:harmonic_expression} below, i.e., to construct sequences  $\{(a_n,b_n)\}_{n\in\mathbb{Z}}\subset \mathcal G$ such that, setting
\begin{equation}
\label{eq:def_a_alpha_b_beta}
    (\alpha_n,\beta_n):= \left(e^{a_n},e^{b_n}\right)\in\mathcal{K},
\end{equation}
the associated series 
\begin{equation}\label{eq: h(i,j)}
        h(i,j):=\sum_{n\in\mathbb{Z}}\left(\alpha_n^i\beta_n^j - \alpha_{n+1}^i\beta_n^j\right)
    \end{equation}
is well defined, positive and satisfies \eqref{eq:harmonicity_relation_interior}--\eqref{eq:Dirichlet_V}.

To that purpose, we introduce a few useful notations. We first define
\begin{equation}
\label{eq:def_G_0}
	\mathcal{G}_0:=\{(x,f(x)):x\in(x_0,0]\}\cup \{(g(y),y):y\in(y_0,0]\}\subset\mathcal{G},
\end{equation}
which corresponds to the yellow part on Figure~\ref{fig:MB_and_escargot}. We first set
\begin{equation*}
   \widehat{f}:=\bigl(f_{\vert(-\infty,x_0]}\bigr)^{-1}\qquad \text{and}\qquad \widetilde{f}:=\bigl(f_{\vert(0,x_0]}\bigr)^{-1}.
\end{equation*}
Since $f$ is concave, strictly increasing on $(-\infty,x_0]$, then $\widehat{f}$ is a well-defined function, convex and strictly increasing on $(-\infty,f(x_0)]$. Similarly, the function $\widehat{g}:=(g_{\vert(-\infty,y_0]})^{-1}$ is convex, strictly increasing on $(-\infty,g(y_0)]$.

We construct a sequence $\{(a_n,b_n)\}_{n\in\mathbb{Z}}$ contained in $\mathcal{G}$ as follows:
    \begin{align}
     (a_0,b_0) &\in \mathcal{G}_0,\label{eq:recursion_0_n}\\
     (a_n,b_n)&:= \left(\widehat{f}\circ(\widehat{g}\circ\widehat{f})^{\circ(n-1)}(b_0),(\widehat{g}\circ\widehat{f})^{\circ(n)}(b_0)\right),& n\geq 1,\label{eq:recursion_positive_n}\\
     (a_{-n},b_{-n})&:= \left( (\widehat{f}\circ\widehat{g})^{\circ(n)}(a_0), \widehat{g}\circ(\widehat{f}\circ\widehat{g})^{\circ(n-1)}(a_0) \right),& n\geq 1.\label{eq:recursion_negative_n}
\end{align}


\begin{figure}
\begin{center}
\includegraphics[width=0.36\textwidth]{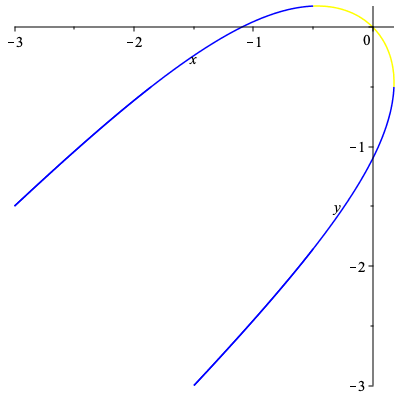}\qquad
\includegraphics[width=0.36\textwidth]{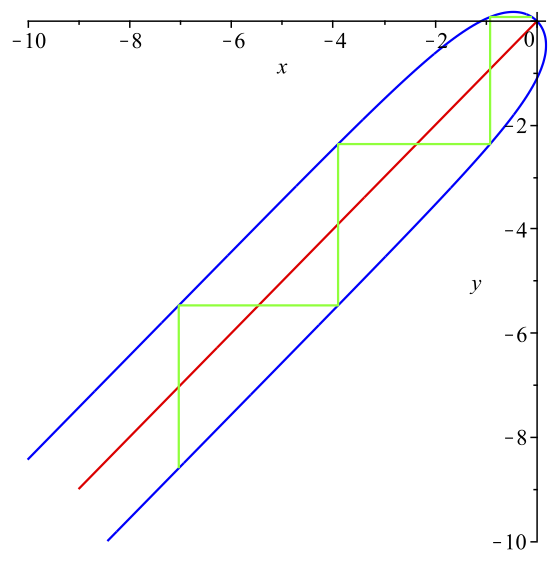}
\end{center}
\caption{Left: an example of curve $\mathcal G$ (blue color) and the subpart (yellow) $\mathcal G_0$ corresponding to the Martin boundary. Right: an example of the compensation approach procedure.}
\label{fig:MB_and_escargot}
\end{figure}

\begin{lem}\label{lem:harmonic_expression}
    Let $\{(a_n,b_n)\}_{n\in\mathbb Z}$ be any sequence defined by  \eqref{eq:recursion_0_n}, \eqref{eq:recursion_positive_n} and \eqref{eq:recursion_negative_n}, and set $(\alpha_n,\beta_n)=(e^{a_n},e^{b_n})$ as in \eqref{eq:def_a_alpha_b_beta}. Then the series given by \eqref{eq: h(i,j)} converges absolutely for any $i,j\geq 1$ and satisfies Eq.~\eqref{eq:harmonicity_relation_interior}--\eqref{eq:Dirichlet_V}. Furthermore, $h(i,j)\geq 0$ for any $i,j\geq 1$.
\end{lem}

\begin{proof}
We start by proving the convergence of the series. To that aim, we construct an auxiliary sequence ${(\widehat{a}_n,\widehat{b}_n)}_{n\in\mathbb{Z}}$ as follows:
    \begin{align*}
         (\widehat{a}_0,\widehat{b}_0) &= (a_0,b_0)\in\mathcal{G}_0,\\
         (\widehat{a}_n,\widehat{b}_n) &= \left( b_0 - n c_1 - (n-1) c_2,b_0 - n c_1 - n c_2\right), &n\geq 1,\\
         (\widehat{a}_{-n},\widehat{b}_{-n}) &= \left( a_0 - n c_1 - n c_2,a_0 - (n-1) c_1 - n c_2\right), &n\geq 1,
\end{align*}
where we have set $c_1 = f(x_0)-x_0>0$ and $c_2 = g(y_0)-y_0>0$, with $x_0,y_0$ as in Lemma~\ref{lem:definition_fg}. Let us also put, for any $n\in\mathbb{Z}$,
    \begin{equation*}
        \bigl( \widehat{\alpha}_n,\widehat{\beta}_n \bigr) = \bigl( e^{\widehat{a}_n},e^{\widehat{b}_n} \bigr).
    \end{equation*}
Since $\mathcal{G}^+$ (see \eqref{eq:def_G+-}) is convex and contains the ray $\{t(1,1):t<0\}$, it should also contain the rays $\{(x_0,f(x_0))+t(1,1):t<0\}$ and $\{(g(y_0),y_0)+t(1,1):t<0\}$. This implies that
    \begin{equation}\label{eq: graph inequalities}
        \widehat{f}(y) \leq y - c_1\quad\text{and}\quad
        \widehat{g}(x) \leq x - c_2,
    \end{equation}
    for all $y\leq f(x_0)$ and $x\leq g(y_0)$. Thanks to the above inequalities and the monotonicity of $\widehat{g}$, we first have
    \begin{equation*}
        a_1\leq \widehat{a}_1\quad\text{and}\quad b_1 = \widehat{g}(a_1)\leq \widehat{g}(\widehat{a}_1) \leq \widehat{b}_1.
    \end{equation*}
    By induction argument, we then have
    \begin{equation*}
        a_n= \widehat{f}(b_{n-1})\leq \widehat{f}(\widehat{b}_{n-1}) \leq \widehat{a}_n\quad\text{and}\quad
        b_n= \widehat{g}(a_{n-1})\leq \widehat{g}(\widehat{a}_{n-1}) \leq \widehat{b}_n,
    \end{equation*}
    for any $n\geq 2$. Similarly, $a_{-n}\leq \widehat{a}_{-n}$ and $b_{-n}\leq  \widehat{b}_{-n}$ for any $n\geq 1$. We now have
    \begin{multline*}
        \sum_{n\in\mathbb{Z}} (\alpha_n^i\beta_n^j + \alpha_{n+1}^i\beta_n^j)
        \leq  \sum_{n\in\mathbb{Z}} (\widehat{\alpha}_n^i\widehat{\beta}_n^j + \widehat{\alpha}_{n+1}^i\widehat{\beta}_n^j)=\widehat{\alpha}_0^i\widehat{\beta}_{-1}^j+\widehat{\alpha}_0^i\widehat{\beta}_0^j+ \widehat{\alpha}_1^i\widehat{\beta}_0^j+\\
        \sum_{n\geq 1} ( e^{-jc_2}+e^{jc_1} )\alpha_0^{i+j}\left(e^{-(c_1+c_2)(i+j)}\right)^n +\sum_{n\geq 1} ( e^{ic_2}+e^{-ic_1} )\beta_0^{i+j}\left(e^{-(c_1+c_2)(i+j)}\right)^n.
    \end{multline*}
The last sums are finite for any $i,j\geq 1$. Hence $h(i,j)$ converges absolutely for any $i,j\geq 1$.
    
We now prove that the series $h(i,j)$ satisfies Eq.~\eqref{eq:harmonicity_relation_interior}--\eqref{eq:Dirichlet_V}. Firstly, all the terms $\alpha_n^i\beta_n^j$ and $\alpha_{n+1}^i\beta_n^j$ satisfy Eq.~\eqref{eq:harmonicity_relation_interior} (since $(\alpha_n,\beta_n)$ and $(\alpha_{n+1},\beta_n)\in\mathcal{K}$), then $h(i,j)$ should also satisfy Eq.~\eqref{eq:harmonicity_relation_interior}. Secondly, since
\begin{equation*}
    h(i,0) = \sum_{n\in\mathbb{Z}}(-\alpha_n^i + \alpha_n^i) = 0\quad\text{and}\quad h(0,j) = \sum_{n\in\mathbb{Z}}(\beta_n^j - \beta_n^j) = 0,
\end{equation*}
then $h(i,j)$ satisfies Eq.~\eqref{eq:Dirichlet_H} and \eqref{eq:Dirichlet_V}.
    
    We now prove that $h(i,j)\geq 0$ for all values of $i,j\geq 1$, by rewriting
    \begin{equation*}
        h(i,j) =\alpha_0^i\beta_0^j\left(1- \left(\frac{\beta_{-1}}{\beta_0}\right)^j- \left(\frac{\alpha_1}{\alpha_0}\right)^i\right) + \sum_{n\geq 1}\alpha_{-n}^i\bigl( \beta_{-n}^j- \beta_{-(n+1)}^j\bigr) + \sum_{n\geq 1}\bigl( \alpha_n^i- \alpha_{n+1}^i\bigr)\beta_n^j.
    \end{equation*}
    The first term is positive for any $i,j\geq 1$ such that $i+j$ is large enough, whereas the second and third terms are positive since $\alpha_{-n}>\beta_{-(n+1)}>\alpha_{-(n+1)}$ and $\beta_n>\alpha_{n+1}>\beta_{n+1}$ for any $n\geq 0$. Hence, there exists $n_0>0$ such that $h(i,j)>0$ for all $i,j\geq 1$ and $i+j\geq n_0$. Thus,
    \begin{equation*}
        h(i,j) = \sum_{k,\ell} p_{k,\ell}h(i+k,j+\ell) \geq 0,
    \end{equation*}
    for any $i,j\geq 1$ and $i+j=n_0-1$. By induction argument, we have $h(i,j)\geq 0$ for all $i,j\geq 1$. The proof is then complete.
\end{proof}

The sign of harmonic funtions $h(i,j)$ can be asserted more specifically.

\begin{rem}\label{rem: positive harmonic func}
     If the walk has jumps inside the positive quadrant, i.e., there exist $k,\ell \geq  0$ such that $k+\ell>0$ and $p_{k,\ell}>0$, then $h(i,j)>0$ for all $i,j\geq 1$. On the other hand, if the walk does not have any jumps inside the positive quadrant, i.e., $p_{k,\ell}=0$ for all $k,\ell \geq 0$, then $h(i,j)>0$ for $i,j\geq 1$ and $i+j$ large enough, but $h(i,j)$ can vanish for some small $i,j$.  
\end{rem}

\subsection{Dependency on the starting point}
\label{subsec:starting}

In Lemma~\ref{lem:harmonic_expression}, we assumed that the starting point $(a_0,b_0)$ belongs to $\mathcal G_0$. In this subsection, we prove that we do not obtain more harmonic functions with starting points in $\mathcal G\setminus \mathcal G_0$.

We construct a sequence $\{(a_n,b_n)\}_{n\in\mathbb{Z}}$ as follows:
\begin{align}
    & (a_0,b_0)\in\mathcal{G},\label{eq: a_0, b_0 in mathcal_G}\\
    & a_n\not= a_{n-1} \text{ s.t. } (a_n,b_{n-1})\in\mathcal{G} , \quad b_n\not= b_{n-1} \text{ s.t. } (a_n,b_n)\in\mathcal{G},\label{eq: a_n, b_n, n>0}\\
    & b_{-n}\not= b_{-(n-1)} \text{ s.t. } (a_{-(n-1)},b_n)\in\mathcal{G},\quad a_{-n}\not= a_{-(n-1)} \text{ s.t. } (a_{-n},b_{-n})\in\mathcal{G},\quad n\geq 1.\label{eq: a_n, b_n, n<0}
\end{align}

Let us define
\begin{equation*}
    \mathcal{O}(\mathcal{G}_0):= \{ \{(a_n,b_n)\}_{n\in\mathbb{Z}}\text{ satisfying \eqref{eq: a_0, b_0 in mathcal_G}--\eqref{eq: a_n, b_n, n<0} s.t. } (a_0,b_0)\in\mathcal{G}_0\}.
\end{equation*}
In other words, $\mathcal{O}(\mathcal{G}_0)$ is the set of orbits formed by the sequences $\{(a_n,b_n)\}_{n\in\mathbb{Z}}$ satisfying \eqref{eq: a_0, b_0 in mathcal_G}--\eqref{eq: a_n, b_n, n<0} and initiated from a point in $\mathcal{G}_0$.

We have the following lemma.
\begin{lem}
    Any sequence $\{(a_n,b_n)\}_{n\in\mathbb{Z}}$ satisfying \eqref{eq: a_0, b_0 in mathcal_G}--\eqref{eq: a_n, b_n, n<0} belongs to the set $\mathcal{O}(\mathcal{G}_0)$. As a result, the series $h(i,j)$ defined in \eqref{eq: h(i,j)} does not depend on the initial value $(\alpha_0,\beta_0)$, but only depends on the orbits in $\mathcal{O}(\mathcal{G}_0)$. 
\end{lem}

\begin{proof}
    Assume that $\{(a_n,b_n)\}_{n\in\mathbb{Z}}$ satisfies \eqref{eq: a_0, b_0 in mathcal_G}--\eqref{eq: a_n, b_n, n<0} with $(a_0,b_0)\notin \mathcal{G}_0 $. It is sufficient to show that there always exists $n_0\in\mathbb{Z}$ such that $(a_{n_0},b_{n_0})\in\mathcal{G}_0$ or $(a_{n_0},b_{n_0-1})\in\mathcal{G}_0$.
    
    If $0\leq a_0< g(y_0)$, then $(a_0,b_{-1})=\bigl(a_0,(g|_{[y_0,0]})^{-1}(a_0)\bigr)\in\mathcal{G}_0$. Similarly, if $0\leq b_0< f(x_0)$, then $(a_1,b_0) =\bigl((f\vert_{[x_0,0]})^{-1}(b_0)\bigr) \in\mathcal{G}_0$.
    
    If $x_0<a_0<0$, then $(a_0,b_{-1})=(a_0,f(a_0))\in\mathcal{G}_0$. If $y_0<b_0<0$, $(a_1,b_0)=(g(b_0),b_0)\in\mathcal{G}_0$.
    
    If $a_0<x_0$ and $b_0<y_0$, without loss of generality, we will further assume that $b_0 = f(a_0)$. Since
    \begin{equation*}
        f(x) \geq x + c_1\quad\text{and}\quad
        g(y) \geq y + c_2,
    \end{equation*}
    for all $x\leq x_0$ and $y\leq y_0$, then there exists $n_0\geq 1$ such that $a_{n_0}=(g\circ f)^{\circ (n)}(a_0)>x_0$ or $b_{n_0}=f\circ(g\circ f)^{\circ (n)}(a_0)>y_0$. Hence, $a_{n_0}$ or $b_{n_0}$ falls into the preceding cases. The rest of the proof then follows trivially.
\end{proof}

\subsection{Behavior on the boundary}
\label{subsec:boundary}

We will show that the harmonic functions defined in \eqref{eq: h(i,j)}, considered as a function of $ (a_0,b_0)\in\mathcal{G}_0$,  will converge (after normalization) to a non-trivial positive harmonic function as $(a_0,b_0)$ tends to the boundary of $\mathcal{G}_0$, which consists of the points $(x_0,f(x_0))$ and $(g(y_0),y_0)$, with $x_0$ and $y_0$ introduced in Lemma~\ref{lem:definition_fg}.

For that matter, let us re-introduce the sequence $\{(a_n,b_n)\}_{n\in\mathbb{Z}}$ in \eqref{eq:recursion_0_n}--\eqref{eq:recursion_negative_n} as functions of one variable $y$:
\begin{align*}
	&b_0(y)=y\in (y_0,0),\quad a_0(y) = g(y),\quad b_{-1}(y) = \widehat{g}\circ g (y),\\
	&(a_{n}(y),b_{n}(y))= \left( \widehat{f}\circ(\widehat{g}\circ\widehat{f})^{\circ(n-1)}(y), (\widehat{g}\circ\widehat{f})^{\circ(n)}(y) \right),& n\geq 1,\\
	&(a_{-n}(y),b_{-(n+1)}(y))= \left( \widehat{f}\circ(\widehat{g}\circ\widehat{f})^{\circ(n-1)}\circ b_{-1}(y), (\widehat{g}\circ\widehat{f})^{\circ(n)}\circ b_{-1}(y) \right),& n\geq 1.
\end{align*}
We further denote
\begin{align*}
	(\alpha_n(y),\beta_n(y)) &= \left(e^{a_n(y)},e^{b_n(y)}\right),\quad n\in\mathbb{Z},\\
	 h_y(i,j) &= \sum_{n\geq 0}  \alpha_n(y)^i\beta_n(y)^j - \alpha_{n+1}(y)^i\beta_n(y)^j  ,
\end{align*}
for all $i,j\geq 0$, $y\in (y_0,0)$. Following the above construction, we remark that
\begin{equation*}
	\bigl(a_{-n}(y),b_{-(n+1)}(y)\bigr) = \bigl(a_n\circ b_{-1}(y),b_n\circ b_{-1}(y)\bigr)
\end{equation*}
for any $n\geq 1$. Then, putting for any $i,j\geq 0$ and $y\in(y_0,0)$
\begin{equation*}
	t_{i,j}(y) = \sum_{n\geq 0}  \alpha_{n}(y)^i\beta_{n}(y)^j - \alpha_{n+1}(y)^i\beta_{n}(y)^j,
\end{equation*}
one has
\begin{equation*}
	h_y(i,j) = t_{i,j}(y) - t_{i,j}(b_{-1}(y)).
\end{equation*}
The following proposition shows the convergence of harmonic functions $h_y(i,j)$ as $y\to y_0^+$.

\begin{prop}
\label{prop:boundary_expression_hf}
	For any $i,j\geq 0$, 
	\begin{equation*}
		\lim_{y\to y_0^+} h_y(i,j) = 0.
	\end{equation*}
	Moreover, for any $i,j\geq 0$, $t_{i,j}(y)$ is differentiable on $[y_0,0)$ and
	\begin{equation*}
		\lim_{y\to y_0^+}\frac{h_y(i,j)}{y-y_0} = 2t_{i,j}' (y_0).
	\end{equation*}
	Furthermore, $t_{i,j}'(y_0)\geq 0$ for all $i,j\geq 1$.
\end{prop}

The remainder of this section is devoted to proving Proposition~\ref{prop:boundary_expression_hf}. We first study the derivatives of some important functions.

\begin{lem}\label{lem: derivatives of some functions}
	We have:
	\begin{enumerate}
		\item\label{item: b1'} $\lim_{y\to y_0^+}b_{-1}(y) = y_0$ and $b_{-1}'(y_0) = -1$;
		\item\label{item: f hat'} $\widehat{f}'$ is strictly positive and strictly increasing on $(-\infty,f(x_0))$. Furthermore, \begin{equation*}
		    \lim_{y\to-\infty}\widehat{f}'(y) = 1;
		\end{equation*}
		\item\label{item: g hat'} $\widehat{g}'$ is strictly positive and strictly increasing on $(-\infty,g(y_0))$. Furthermore, \begin{equation*}
		\lim_{x\to-\infty}\widehat{g}'(x) = 1.
		\end{equation*}
	\end{enumerate}
\end{lem}

\begin{proof}
	We first prove Item~\ref{item: b1'}. Since $g(y)$ is strictly concave on $(-\infty,0)$ and admits its maximum value at $y_0$, then $g'(y_0) = 0$ and
	\begin{equation*}
		g(y) = g(y_0) + \frac{g''(y_0)}{2} (y-y_0)^2 + o((y-y_0)^2)
	\end{equation*}
	as $y\to y_0$, where $g''(y_0)<0$. Hence,
	\begin{equation*}
		\lim_{y\to y_0^+} \left(\frac{b_{-1}(y)-y_0}{b_0(y)-y_0}\right)^2 =\lim_{y\to y_0^+} \frac{g(b_{-1}(y))-g(y_0)}{g(b_0(y))-g(y_0)} = 1.
	\end{equation*}
	This implies $b_{-1}(y)\to y_0$ as $b_0(y)=y\to y_0^+$. Since $b_{-1}(y)<y_0<b_0(y)$ for all $y\in (y_0,0)$, then $b_{-1}'(y_0)=-1$.
	
	We move to the proof of Item~\ref{item: f hat'} (Item~\ref{item: g hat'} would be proven similarly). Since $\widehat{f}$ is strictly convex and strictly increasing on $(-\infty,f(x_0))$, then $\widehat{f}'$ is strictly increasing and positive on $(-\infty,f(x_0))$. Putting
	\begin{equation*}
	\alpha(\beta) = e^{\widehat{f}(\log\beta)},\quad \beta\geq 0,
	\end{equation*}
	we first study the behavior of $\alpha'(\beta)$ as $\beta\to 0$. Since $\alpha(\beta)$ satisfies $K(\alpha(\beta),\beta)=0$, one may differentiate once and twice the equation and evaluate them at $\beta=0$ as follows:
	\begin{align*}
	& \partial_\alpha K(0,0)\alpha'(0)+\partial_\beta K(0,0)  = 0,\\
	& \partial_{\alpha\alpha} K(0,0)\alpha'(0)^2 + 2\partial_{\alpha\beta} K(0,0) \alpha'(0) + \partial_{\beta\beta} K(0,0) + \partial_\alpha K(0,0)\alpha''(0) = 0.
	\end{align*}
	While the value of $\alpha'(0)$ cannot be deduced from the first equation,  $\alpha'(0)$ can be solved explicitly from the second:
	\begin{equation*}
	\alpha' (0) = \frac{1-p_{0,0}\pm \sqrt{(1-p_{0,0})^2-4p_{1,-1}p_{-1,1}}}{2p_{1,-1}}\not= 0,\infty.
	\end{equation*}
	These solutions correspond to the behavior of the two branches $(e^{\widehat{f}(y)},e^y)$ and $(e^{g(y)},e^y)$ as $y\to-\infty$. This implies that
	\begin{equation*}
	\widehat{f}'(y) = \alpha'(e^y)\frac{e^y}{e^{\widehat{f}(y)}}\overset{y\to-\infty}{\longrightarrow}\alpha' (0)\frac{1}{\alpha' (0)} = 1.\qedhere
	\end{equation*}
\end{proof}

The following lemma shows the convergence of $t_{i,j}'(y)$.
\begin{lem}
	For any $i,j\geq 1$ and $y\in [y_0,0)$, the series
	\begin{equation*}
		t_{i,j}'(y) = \sum_{n\geq 0} \bigl( \alpha_n(y)^i \beta_n (y)^j - \alpha_{n+1}(y)^i \beta_n (y)^j \bigr)'
	\end{equation*}
	converges absolutely.
\end{lem}

\begin{proof}
	By Lemma~\ref{lem: derivatives of some functions}, the derivatives of $\alpha_n(y)$ and $\beta_n(y)$ are well defined on $[y_0,0)$ for all $n\geq 0$.
	Since $\widehat{f}(y),\widehat{g}(x)\downarrow 1$ as $x,y\to-\infty$, then
	\begin{align*}
	&a_{n}'(y) = \widehat{f}\circ b_{n-1}(y) \prod_{k=1}^{n-1}\widehat{g}'\circ a_{k}(y) \cdot \widehat{f}'\circ b_{k-1}(y)=o(r^n),\\
	&b_{n}'(y) = \prod_{k=1}^{n}\widehat{g}'\circ a_{k}(y) \cdot \widehat{f}'\circ b_{k-1}(y)=o(r^n),
	\end{align*}
	as $n\to\infty$, for any $r>1$. Let us now fix $i,j\geq 1$. Recall from the proof of Lemma~\ref{lem:harmonic_expression} that
	\begin{equation*}
		\alpha_n(y)^i\beta_n(y)^j \leq  y^{i+j}e^{ic_2}\left(e^{-(i+j)(c_1+c_2)}\right)^n = O \left(e^{-(i+j)(c_1+c_2)}\right)^n,
	\end{equation*}
	as $n\to\infty$, where $c_1$ and $c_2$ are positive constants. Hence,
	\begin{equation*}
			\bigl\vert \bigl(\alpha_n(y)^i \beta_n (y)^j \bigr)'\bigr\vert = \bigl\vert \bigl(ia_n'(y)+ jb_n'(y)\bigr) \alpha_n(y)^i \beta_n (y)^j \bigr\vert \leq o(r^n),
	\end{equation*}
	as $n\to\infty$, for any $y\in[y_0,0)$ and $r>0$. Similarly, we also have
	\begin{equation*}
		\bigl\vert \bigl(\alpha_{n+1}(y)^i \beta_n (y)^j \bigr)' \bigr\vert \leq o(r^n),
	\end{equation*}
	as $n\to\infty$, for any $y\in[y_0,0)$ and $r>0$. Thus, the series
	\begin{equation*}
	    t_{i,j}'(y) = \sum_{n\geq 0}\bigl( ia_n'(y)+jb_n'(y)\bigr)\alpha_n(y)^i \beta_n (y)^j - \bigl(ia_{n+1}'(y)+jb_n'(y)\bigr)\alpha_{n+1}(y)^i \beta_n (y)^j
	\end{equation*}
	converges absolutely for any $i,j\geq 1$.
\end{proof}

We now prove Proposition~\ref{prop:boundary_expression_hf}.
\begin{proof}[Proof of Proposition~\ref{prop:boundary_expression_hf}]
	Since $b_{-1}(y)\to y_0$ as $y\to y_0^+$ by Lemma~\ref{lem: derivatives of some functions}\ref{item: b1'}, then
	\begin{equation*}
		h_y(i,j) = t_{i,j}(y) - t_{i,j}(b_{-1}(y)) \to 0,
	\end{equation*}
	as $y\to y_0^+$. By L'H\^opital's rule, we now deduce
	\begin{equation*}
		\frac{h_y(i,j)}{y-y_0} \overset{y\to y_0^+}{\longrightarrow} t_{i,j}'(y_0) - t_{i,j}'(b_{-1}(y_0))b_{-1}'(y_0) = 2t_{i,j}'(y_0).	
	\end{equation*}
	Since $h_y(i,j)$ is harmonic and non-negative for all $y\in(y_0,0)$, this guarantees that $t_{i,j}'(y_0)$ is harmonic and non-negative for all $i,j\geq 1$.
	
	We now prove that $t_{i,j}' (y_0)>0$ for all $i,j\geq 1$ and $i+j$ large enough. Since for any $n\geq 0$ one has $\alpha_n(y_0) > \alpha_{n+1}(y_0)$, then
	\begin{equation}\label{eq: alpha_n/alpha_n+1_1}
	    jb_n'(y_0)\alpha_n(y_0)^i\beta_n(y_0)^j - jb_n'(y_0)\alpha_{n+1}(y_0)^i\beta_n(y_0)^j >0,
	\end{equation}
	for any $n\geq 0$ and $i,j\geq 1$. We then show that
	\begin{equation}\label{eq: alpha_n/alpha_n+1_2}
	    ia_n'(y_0)\alpha_n(y_0)^i\beta_n(y_0)^j - ia_{n+1}'(y_0)\alpha_{n+1}(y_0)^i\beta_n(y_0)^j >0,
	\end{equation}
	or equivalently,
	\begin{equation*}
	    \alpha_n(y_0)^i > \alpha_{n+1}(y_0)^i\bigl(\widehat{f}'\circ b_n(y_0)\bigr)\cdot\bigl( \widehat{g}'\circ a_n(y_0)\bigr),
	\end{equation*}
	for any $n\geq 0$, $j\geq 1$ and $i$ large enough. Indeed, recall that
	\begin{equation*}
	    \frac{\alpha_n(y_0)}{\alpha_{n+1}(y_0)} = e^{a_n(y_0)-a_{n+1}(y_0)} \geq e^{f(x_0)-x_0 + g(y_0)-y_0} >0,
	\end{equation*}
	for any $n\geq 0$, and $(\widehat{f}'\circ b_n(y_0))\cdot (\widehat{g}'\circ a_n(y_0))\downarrow 1$  as $n\to\infty$. Hence, for any $n\geq 0$ and $i$ large enough,
	\begin{equation*}
	    \left(\frac{\alpha_n(y_0)}{\alpha_{n+1}(y_0)}\right)^i > \bigl(\widehat{f}'\circ b_n(y_0) \bigr)\cdot \bigl(\widehat{g}'\circ a_n(y_0)\bigr),
	\end{equation*}
	which is equivalent to Inequality~\eqref{eq: alpha_n/alpha_n+1_2}. Both \eqref{eq: alpha_n/alpha_n+1_1} and \eqref{eq: alpha_n/alpha_n+1_2} imply that there exists $n_0>1$ such that $t_{i,j}'(y_0)>0$ for any $j\geq 1$ and $i\geq n_0$. The harmonicity then implies that $t_{i,j}'(y_0)>0$ for any $i,j\geq 1$ and $i+j\geq n_0+1$.
\end{proof}

We want to point out that Remark~\ref{rem: positive harmonic func} also holds true for the harmonic function $t_{i,j}(y_0)$, that is, $t_{i,j}(y_0)>0$ for all $i,j\geq 1$ if there exist $k,\ell \geq 0$ such that $k+\ell\geq 1$ and $p_{k,\ell}>0$.

\section{Explicit expression for walks with small steps}
This section aims at giving an explicit expression for all terms appearing in the harmonic functions \eqref{eq: h(i,j)}, in the case where the random walks only have small jumps, that is, the positive transition probabilities can only be in the set $\{ p_{-1,1}, p_{1,-1}, p_{1,0}, p_{0,1}, p_{1,1} \}$, see Figure~\ref{fig:step_sets}.

The following lemma presents a uniformization of the zero set of $K(\alpha,\beta)$. 
\begin{lem}
\label{lem: uniformization}
	One has
	\begin{equation*}
		\{(\alpha,\beta)\in\mathbb{C}^2:K(\alpha,\beta)=0\} = \left\{(\alpha(s),\beta(s)):s\in\mathbb{C}\right\},
	\end{equation*}
	where
	\begin{equation*}
		\alpha(s)^{-1} := \frac{\sqrt{b^2-ac}}{2a}\left(s+\frac{1}{s}\right)+\frac{b}{a} \quad \text{and}\quad
		\beta(s)^{-1} := \frac{\sqrt{{\widehat b^2}-a\widehat{c}}}{2a}  \left(\rho s +  \frac{1}{\rho s} \right) + \frac{\widehat{b}}{a},
	\end{equation*}
	and $a:= 1-4p_{-1,1}p_{1,-1}$, $\rho:= \sqrt{\frac{1+\sqrt{a}}{1-\sqrt{a}}}$, and finally
	\begin{align*}
		&b:= p_{0,1}+2p_{-1,1}p_{1,0},
		&c:=  p_{0,1}^2 - 4p_{-1,1}p_{1,1},\\
		& \widehat{b}:= p_{1,0}+2p_{1,-1}p_{0,1}, &\widehat{c}:=p_{1,0}^2 - 4p_{1,-1}p_{1,1}.
	\end{align*}
	We further have the involutions $\alpha(s)=\alpha(1/s)$ and $\beta(s) = \beta(1/(\rho^2 s))$. 
\end{lem}

\begin{proof}
	We will find a rational uniformization of the algebraic curve
	\begin{equation}\label{eq: K(1/alpha,1/beta)=0}
		\alpha^2\beta^2 K\left(\frac{1}{\alpha},\frac{1}{\beta}\right)=0.
	\end{equation}
	We first rewrite
	\begin{align*}
		\alpha^2\beta^2 K\left(\frac{1}{\alpha},\frac{1}{\beta}\right)&= p_{1,1} + p_{1,0}\beta + p_{0,1}\alpha + p_{1,-1}\beta^2 + p_{-1,1}\alpha^2 - \alpha\beta\\
		&= \left( p_{1,-1} \right)\beta^2 + \left( -\alpha + p_{1,0}  \right)\beta + \left( p_{-1,1}\alpha^2 + p_{0,1}\alpha + p_{1,1} \right).
	\end{align*}
	The above polynomial admits the discriminant
	\begin{align*}
		\delta(\alpha)&:= \left( -\alpha + p_{1,0}  \right)^2 - 4p_{1,-1}\left( p_{-1,1}\alpha^2 + p_{0,1}\alpha + p_{1,1} \right)\\
		&= a\alpha^2 -2b\alpha +c = \frac{b^2-ac}{4a}\left[\left(2\frac{a\alpha-b}{\sqrt{b^2-ac}}\right)^2 -4  \right],
	\end{align*}
	where $a,b,c$ are defined in Lemma~\ref{lem: uniformization}, and
	\begin{equation*}
	    b^2-ac = 4p_{1,-1}\bigl( (1-4p_{-1, 1}p_{1, -1})p_{1, 1}+p_{1, 0}^2 p_{-1, 1} + p_{0, 1}^2 p_{1, -1}+p_{0, 1} p_{1, 0} \bigr) >0.
	\end{equation*}
	By setting
	\begin{equation*}
		2\frac{a\alpha-b}{\sqrt{b^2-ac}} = s+\frac{1}{s},\quad\text{that is,}\quad \alpha = \frac{\sqrt{b^2-ac}}{2a}\left(s+\frac{1}{s}\right)+\frac{b}{a},\quad s\in\mathbb{C},
	\end{equation*}
	one can easily find $\beta$ satisfying Eq.~\eqref{eq: K(1/alpha,1/beta)=0} as follows:
	\begin{equation*}
		\beta = \frac{\alpha(s)-p_{0,1}+\sqrt{\delta\circ\alpha(s)}}{2p_{1,-1}}=\frac{\sqrt{\widehat{b}^2-a\widehat{c}}}{2a}  \left(\rho s +  \frac{1}{\rho s} \right) + \frac{\widehat{b}}{a},
	\end{equation*}
	where $\widehat{b}, \widehat{c}, \rho$ are defined in Lemma~\ref{lem: uniformization}. The proof is then complete.
\end{proof}

By the above lemma, one can describe the curve $\mathcal{K}$ in \eqref{eq:curve_K} as
\begin{equation*}
    \mathcal{K} = \{ (\alpha(s),\beta(s)) : s>0 \}.
\end{equation*}

We remark that at $s=1$, the discriminant $\delta\circ\alpha(1) =0$, then $\beta(1)$ is the double root of Eq.~\eqref{eq: K(1/alpha,1/beta)=0} as an equation of $\beta$. Similarly, at $s=1/\rho$, $\alpha(1/\rho)$ is a double root of Eq.~\eqref{eq: K(1/alpha,1/beta)=0} as an equation of $\alpha$. Thus, defining $\mathcal{K}_0$ as the analogue of $\mathcal G_0$ (see \eqref{eq:def_G_0}) before the exponential change of variable, one easily deduces the description
\begin{equation*}
    \mathcal{K}_0=\{(\alpha(s),\beta(s)):s\in(1/\rho,1)\}.
\end{equation*}

\begin{prop}\label{prop: alpha_n,beta_n small step walk}
    Any sequence $\{(\alpha_n,\beta_n)\}_{n\in\mathbb{Z}}$ satisfying \eqref{eq: a_0, b_0 in mathcal_G}--\eqref{eq: a_n, b_n, n<0} and \eqref{eq:def_a_alpha_b_beta}, with $(\alpha_0,\beta_0)\in\mathcal{K}_0$, can be described explicitly as
    \begin{equation*}
	    \{(\alpha_n,\beta_n)\}_{n\in\mathbb{Z}}=\left\{\left(\alpha(\rho^{2n}s),\beta(\rho^{2n}s)\right)\right\}_{n\in\mathbb{Z}},
    \end{equation*}
    with $s\in(1/\rho,1)$. As a consequence, the harmonic function formed by $(\alpha_n,\beta_n)_{n\in\mathbb{Z}}$ in \eqref{eq: h(i,j)} becomes
    \begin{equation*}
	    h(i,j) = \sum_{k\in\mathbb{Z}}  \alpha(\rho^{2n}s)^i\beta(\rho^{2n}s)^j-\alpha(\rho^{2n+2}s)^i\beta(\rho^{2n}s)^j,\qquad i,j\geq 1.
\end{equation*}
\end{prop}

\begin{proof}
We will prove the proposition by induction argument. Since $(\alpha_0,\beta_0)\in\mathcal{K}_0$, then
\begin{equation*}
    (\alpha_0,\beta_0) = (\alpha(s),\beta(s)),
\end{equation*}
with $s\in (1/\rho,1)$. Now assume that $(\alpha_n,\beta_n) = (\alpha(\rho^{2n}s),\beta(\rho^{2n}s))$ with $n\in\mathbb{Z}$. Since $\beta_{n} = \beta(\rho^{2n}s) = \beta\left( 1/(\rho^{2n}s)\right)$
by the involution and $(\alpha_{n+1},\beta_{n})$ is a root of $K(\alpha,\beta)$, then
\begin{equation*}
    \alpha_{n+1} = \alpha\left(\frac{1}{\rho^{2n+2}s}\right) = \alpha(\rho^{2n+2}s).
\end{equation*}
Since $(\alpha_{n+1},\beta_{n+1})$ is also a root of $K(\alpha,\beta)$, then
$    \beta_{n+1} = \beta(\rho^{2n+2}s)$.
Similarly, the involutions of Lemma~\ref{lem: uniformization} also imply that
    $\alpha_{n-1} = \alpha(\rho^{2n-2}s)$ and $
    \beta_{n-1} = \beta(\rho^{2n-2}s)$.
The proof is then complete.
\end{proof}

\subsection*{A concrete example} We now return to our simplest walk with the transition probabilities $p_{1,1}=p_{1,-1}=p_{-1,1}=1/3$ and provide a full proof of the formula \eqref{eq:expression_escape_Fibonacci}. The zero set of the associated kernel
\begin{equation*}
    K(\alpha,\beta) = \frac{1}{3}(x^2y^2+x^2+y^2)-xy
\end{equation*}
admits the uniformization presented in Lemma~\ref{lem: uniformization}, with
\begin{equation*}
    \alpha(s) = \frac{\sqrt{5}}{s+1/s},\quad
    \beta(s) = \frac{\sqrt{5}}{\rho s+1/(\rho s)} \quad \text{and} \quad \rho = \frac{3+\sqrt{5}}{2}.
\end{equation*}
The sequence $\{(\alpha_n,\beta_n)\}_{n\in \mathbb{Z}}$ in Proposition~\ref{prop: alpha_n,beta_n small step walk} can be expressed as functions of $s\in(1/\rho,1)$:
\begin{equation*}
    (\alpha_n(s),\beta_n(s)) = \left( \frac{\sqrt{5}}{\rho^{2n}s+1/(\rho^{2n}s)},\frac{\sqrt{5}}{\rho^{2n+1}s+1/(\rho^{2n+1}s)}\right),\qquad n\in\mathbb{Z}.
\end{equation*}
Consequently, this leads to harmonic functions depending on $s$, denoted by
\begin{equation*}
    h_s(i,j):= \sum_{n\in\mathbb{Z}}\alpha_n(s)^i\beta_n(s)^j-\alpha_{n+1}(s)^i\beta_n(s)^j.
\end{equation*}

Let us look at two specific examples. Firstly, the survival probability \eqref{eq:expression_escape_Fibonacci} is the harmonic function associated to the initial point \begin{equation*}
    \left(\alpha_0\left(\frac{\sqrt{5}-1}{2}\right),\beta_0\left(\frac{\sqrt{5}-1}{2}\right)\right)=(1,1)\in\mathcal{K}_0,
\end{equation*}
and can be expressed explicitly as
\begin{equation*}
    \mathbb{P}(\tau_{(i,j)}=\infty) = h_{\frac{\sqrt{5}-1}{2}}(i,j) = \cdots - \frac{1}{5^i 13^j} +\frac{1}{5^i 2^j} -\frac{1}{1^i 2^j} + \frac{1}{1^i 1^j} -\frac{1}{2^i 1^j} + \frac{1}{2^i 5^j} - \frac{1}{13^i 5^j} + \cdots
\end{equation*}
In this formula, the sequence
\begin{equation*}
    \left\{\frac{\rho^ns+\frac{1}{\rho^ns}}{\sqrt{5}}\right\}_{n\geq 0} = \{1, 1, 2, 5, 13, 34, 89, 233, 610, 1597,\ldots\}
\end{equation*}
is a bisection of Fibonacci numbers.

We now give the expression for the normalized positive harmonic function constructed in Proposition~\ref{prop:boundary_expression_hf}. Recall that such a function is obtained by differentiating the harmonic function $h_s(i,j)$ with respect to the variable $\log\beta_0(s)$ and evaluating as $b_0\to \log\beta_0(1)$. (The logarithms appear in our computation because Proposition~\ref{prop:boundary_expression_hf} was stated under $(\log\alpha,\log\beta)$-coordinates instead of $(\alpha,\beta)$-coordinates.) We have:
\begin{equation*}
    \frac{h_s(i,j)}{\log \beta_0(s) - \log \beta_0(1)} \overset{s\to 1^-}{\longrightarrow} \left.\frac{d h_s(i,j)}{ds}\frac{ds}{d\left(\log\beta_0(s)\right)}\right|_{s=1},
\end{equation*}
where
\begin{align*}
    & \left.\frac{ds}{d\left(\log\beta_0(s)\right)}\right|_{s=1} = \frac{\beta_0(1)}{\beta_0'(1)} = -\frac{3}{\sqrt{5}},\\
    &\left.\frac{d h_s(i,j)}{ds}\right|_{s=1} \\
    =& \sum_{n\in\mathbb{Z}}i\left( \alpha_n'(1)\alpha_n(1)^{i-1} - \alpha_{n+1}'(1)\alpha_{n+1}(1)^{i-1} \right)\beta_n(1)^j + j\left(\alpha_n(1)^i-\alpha_{n+1}(1)^i\right)\beta_n'(1)\beta_n(1)^{j-1},\\
    &(\alpha_n(1),\beta_n(1)) = \left(  \frac{\sqrt{5}}{\rho^{2n}+\frac{1}{\rho^{2n}}}, \frac{\sqrt{5}}{\rho^{2n+1}+\frac{1}{\rho^{2n+1}}} \right),\\
    & (\alpha_n'(1),\beta_n'(1)) =  \left(  -\frac{\sqrt{5}\left(\rho^{2n}-\frac{1}{\rho^{2n}}\right)}{\left(\rho^{2n}+\frac{1}{\rho^{2n}}\right)^2},- \frac{\sqrt{5}\left(\rho^{2n+1}-\frac{1}{\rho^{2n+1}}\right)}{\left(\rho^{2n+1}+\frac{1}{\rho^{2n+1}}\right)^2} \right),\quad n\in\mathbb{Z}.
\end{align*}
In these formulas, we want to point out that the sequence
\begin{equation*}
    \left\{\rho^n+\frac{1}{\rho^n}\right\}_{n\geq 0} = \{2, 3, 7, 18, 47, 123, 322, 843, 2207,\ldots\} = \{L_{2n}\}_{n\geq 0}
\end{equation*}
is a bisection of Lucas numbers $\{L_n\}_{n\geq 0}$: 
\begin{equation*}
    L_0 = 2,\quad L_1=1,\quad L_n = L_{n-1} + L_{n-2},\quad n\geq 2,
\end{equation*}
see \href{https://oeis.org/A005248}{A005248}. Moreover, the sequence
\begin{equation*}
    \{u_n\}_{n\geq 0} = \left\{\sqrt{5}\left(\rho^n-\frac{1}{\rho^n}\right)\right\}_{n\geq 0} = \{0, 5, 15, 40, 105, 275, 720, 1885, 4935,\ldots\}
\end{equation*}
satisfies the recurrence relation (see \href{https://oeis.org/A201157}{A201157}):
\begin{equation*}
    u_0 = 0, \quad u_1 = 5, \quad u_n = 3u_{n-1} - u_{n-2},\quad n\geq 2.
\end{equation*}


\part{Green functions and Martin boundary for singular random walks in the quadrant}
\label{part:Martin_boundary}

In this part, we describe the Martin boundary inside the cone and prove that it coincides with the set of harmonic functions constructed in Part~\ref{part:construction}.

\section{Presentation of the results}

\subsection{A brief account on Green functions and Martin boundary theory}
Let us first recall the definition of the Green function of the random walk inside the cone $\mathbb{Z}_{>0}^2$, which plays an important role in the construction of the Martin boundary. For $x,y\in\mathbb{Z}_{>0}^2$, set 
\begin{equation*}
    G(x,y)=\sum_{n\geq 1}\mathbb{P}(x+S(n)=y,\tau_{x}>n)\quad\text{and}\quad \widetilde{G}(x,y)=\sum_{n\geq 1}\mathbb{P}(x+S(n)=y),
\end{equation*}
where $\tau_x=\inf\{n\geq 0: x+S(n)\not\in \mathbb{Z}_{>0}^2\}\leq \infty$, and set $x_0=(1,1)$. The Martin kernel associated to the reference point $x_0$ is then defined as 
\begin{equation}
\label{eq:definition_Martin_kernel}
    K_M(x,y)=\frac{G(x,y)}{G(x_0,y)}.
\end{equation}
The Martin boundary $\partial_M^{\mathbb{Z}_{>0}^2}S$ of the random walk $S$ killed outside of $\mathbb{Z}_{>0}^2$ is the boundary in the topological space $\{f:\mathbb{Z}_{>0}^2\rightarrow \mathbb{R}\}$ (with the topology of point-wise convergence) of the set of maps $\left\{K_M(\,\cdot\,,y): y\in \mathbb{Z}_{>0}^2\right\}$. The space $\partial_M^{\mathbb{Z}_{>0}^2}S$ is a compact measurable subspace of the set of real functions on $\mathbb{Z}_{>0}^2$, see  \cite{Do-59}, and for any non-negative function $h$ on $\mathbb{Z}_{> 0}^2$ which is harmonic with respect to $S$ killed outside $\mathbb{Z}_{> 0}^2$, there exists a positive measure $\mu_h$ on $\partial_M^{\mathbb{Z}_{>0}^2}S$ such that 
\begin{equation*}
    h(\,\cdot\,)=\int_{\partial_M^{\mathbb{Z}_{>0}^2}S}K_M(\,\cdot\,,\omega)d\mu_h(\omega).
\end{equation*}
One sees from \eqref{eq:definition_Martin_kernel} that a convenient way to compute the asymptotics of $K_M(\,\cdot\,,y)$ as $y$ goes to infinity is to first study the behavior of $G(x,y)$ when $y$ goes to infinity in the cone. 

We set $\Sigma=\mathbb{R}_{\geq 0}^2\cap \mathbb{S}^1$, and for $u\in \Sigma$, denote by $\phi(u)$ the unique solution of 
\begin{equation}
\label{eq:new_drift}
    \mathbb{E}\bigl(\exp\langle \phi(u), S(1)\rangle\bigr)=1 \quad \text{and}\quad \mathbb{E}\bigl(S(1)\exp\langle\phi(u), S(1)\rangle\bigr)=r_uu:=\mu^u
\end{equation}
for some $r_u>0$. Remark that $\phi$ is bijective from $\Sigma$ to $\overline{\mathcal{G}_0}$ (whose definition may be found in \eqref{eq:def_G_0}), with
\begin{equation*}
    \phi^{-1}(a,b)=\frac{\nabla K(a,b)}{\Vert \nabla K(a,b)\Vert}
\end{equation*}
for $(a,b)\in \overline{\mathcal{G}_0}$. For $u\in \Sigma$, we then write $\mathbb{P}_u$ for the probability measure on $\{S(n)\}_{n\geq 1}$ given by the transition probabilities 
\begin{equation*}
    \mathbb{P}_u(S(1)\in A)=\mathbb{E}\bigl(\exp\langle\phi(u), S(1)\rangle\mathbf{1}_{S(1)\in A}\bigr)
\end{equation*} 
for $A\subset \mathbb{Z}^2$, and we denote by $\mathbb{E}_u$ the corresponding expectation. We then write $\Sigma^u$ for the covariance matrix of $S(1)$ under $\mathbb{P}_u$ (recall that we denote by $\mu^u$ the drift of $S(1)$ under $\mathbb{P}_u$, see \eqref{eq:new_drift}).

\subsection{Statement of the main result}
The main result of this section is the following theorem, which gives an asymptotics of $G(x,y)$ when $y$ goes to infinity along any direction in the quarter plane. Denote by $e_1=(1,0)$ and $e_2=(0,1)$ the two standard basis vectors.
\begin{thm}\label{thm:asymptotic_inside_green}
For any $x\in \mathbb{Z}_{>0}^2$ and $\{y(n)\}_{n\geq 1}$ with $\vert y(n)\vert\xrightarrow[]{n\rightarrow \infty} \infty$,
\begin{itemize}
\item if $\frac{y(n)}{\vert y(n)\vert}=u_n$ with $\lim u_n:=u_0\in \Sigma\setminus \{e_1,e_2\}$, then
\begin{equation*}
    \lim_{n\rightarrow \infty}\sqrt{\vert y(n)\vert }e^{-\langle\phi(u_n(y(n))),x-y(n)\rangle}G(x,y(n)) =  A(u_0)\mathbb{P}_{u_0}(\tau_x=\infty);
\end{equation*}
\item if $\frac{y(n)}{\vert y(n)\vert}=u_n$ with $\lim u_n=e_i$, $i\in \{1,2\}$, then, with $\bar{\imath}=3-i$,
\begin{equation*}
    \lim_{n\rightarrow \infty}\frac{\vert y(n)\vert^{3/2} e^{-\langle\phi(u(y(n))),x-y(n)\rangle}}{V(y_i(n))}G(x,y(n)) = A(e_i)\bigl(x_{\bar{\imath}}-\mathbb{E}_{e_i}(x_{\bar{\imath}}+S_{\bar{\imath}}(\tau_x))\bigr),
\end{equation*}
\end{itemize}
with $u\mapsto A(u)$ a continuous, positive function on $\Sigma$ and $V$ introduced in Proposition~\ref{prop:equiv_proba_horizon}.
\end{thm}
The proof of this theorem is postponed to Section \ref{Section:proof_main_results}. The main difficulty concerns the case of an asymptotic direction along the boundary axes, where the survival probability vanishes. The latter case is solved by first studying the Green function of the random walk in the half-planes $\mathbb{H}_1:=\mathbb{Z}\times \mathbb{Z}_{>0}$ and $\mathbb{H}_2:=\mathbb{Z}_{>0}\times \mathbb{Z}$. We achieve the latter in Section \ref{Section:halfspace}. To conclude this section, we collect a few useful estimates on the classical random walk on $\mathbb Z^2$.

\subsection{Preliminary estimates, and the Ney and Spitzer theorem}

For all $u\in \Sigma$, we introduce the modified Green kernels
\begin{equation*}
    G_u(x,y)=\sum_{n\geq 1}\mathbb{P}_u(x+S(n)=y,\tau_{x}>n)\quad\text{and}\quad \widetilde{G}_u(x,y)=\sum_{n\geq 1}\mathbb{P}_u(x+S(n)=y).
\end{equation*}
For all $x,y\in \mathbb{Z}_{>0}^2$, we have 
\begin{equation}
\label{eq:relation_Green_function_central_change}
G_u(x,y)=e^{\langle\phi(u), y-x\rangle}G(x,y) \quad\text{and}\quad \widetilde{G}_u(x,y)=e^{\langle\phi(u), y-x\rangle}\widetilde{G}(x,y).
\end{equation}

We recall the following result from Ney and Spitzer \cite[Thm~2]{NeSpi-66}, using our notation $\mu^u=\mathbb{E}_u(S(1))$.
\begin{thm}\label{thm:asymptotic_green_Ney_spiter}
There exists a continuous function $A:\Sigma\rightarrow \mathbb{R}_{>0}$ such that, as $t\to\infty$, uniformly on $u\in \Sigma$ and $x\in \mathbb{Z}^2$ with $\vert x\vert=o(t^{1/2})$,
\begin{equation*}
    \sqrt{t}\widetilde{G}_u(x,\lfloor t\mu^u\rfloor)\rightarrow A(u).
\end{equation*}
\end{thm}
We should emphasize that the initial proof of Ney and Spitzer requires the random walk $\{S(n)\}_{n\geq 1}$ to be irreducible, and is valid only for fixed $x\in \mathbb{Z}^2$. However, the only reason for the first requirement is the use of the local large deviation limit theorem  \cite[Thm~2.1]{NeSpi-66}, which has since been proven for any random walk with finite generating function and whose support generates $\mathbb{Z}^2$, see \cite[Ch.~7, P10]{Spi-64}. Likewise, the hypothesis of a fixed $x$ can be relaxed to the condition $\vert x\vert=o(t^{1/2})$ (see the proof of Proposition \ref{prop:asymptotic_Green_micro} for a similar computation).

We will apply several times the following lemma, which gives a large deviation bound which is uniform on $\mathbb{P}_u$, $u\in \Sigma$.
\begin{lem}
\label{lem:large_deviation_uniform}
There exist constants $\eta,\eta'>0$ such that for all $u\in \Sigma$, $n\geq 1$,
\begin{itemize}
    \item for all $t\in(0,1)$,
\begin{equation*}
    \mathbb{P}_u\bigl(\vert S(n)-n\mu^u\vert> tn\bigr)\leq 4\exp(-\eta t^2n);
\end{equation*}
\item for all $t\geq 1$,
\begin{equation*}
    \mathbb{P}_u\bigl(\vert S(n)-n\mu^u\vert> tn\bigr)\leq 4\exp(-\eta' tn).
\end{equation*}
\end{itemize}
\end{lem}
\begin{proof}
Using Markov inequality, we get (denoting $S(n)=(S_1(n),S_2(n))$)
\begin{equation*}
    \mathbb{P}_u\bigl(\vert S_1(n)-n\mu_1^u\vert> tn\bigr)\leq 2\exp(-n\Lambda^*_{u,1}(t)),
\end{equation*}
where 
$\Lambda^*_{u,1}$ is the Legendre transform of the function 
\begin{equation*}
    \alpha\in\mathbb{R}\mapsto \log\mathbb{E}_u\bigl(e^{\alpha (S_1(1)-\mu_1^u)}\bigr)=\log\mathbb{E}\bigl(e^{(\alpha+\phi(u)_1) S_1(1)}\bigr)-\alpha\mu_1^u.
\end{equation*}
An easy computation yields that 
\begin{equation*}
    \Lambda^*_{u,1}(t)=\Lambda_1(t+\mu_1^u)-\phi(u)_1(t+\mu_1^u),
\end{equation*}
where $\Lambda_1$ is the Legendre transform of $\alpha\mapsto \log\mathbb{E}\bigl(e^{\alpha S_1(1)}\bigr)$. Since $u\mapsto \mu_1^u$ and $u\mapsto \phi(u)_{1}$ are $\mathcal{C}^2$, $u\mapsto\Lambda_{u,1}^*(t)$ is $\mathcal{C}^2$. Since for all $u\in \Sigma$, we have $\Lambda_{u,1}^*(0)=(\Lambda_{u,1}^*)'(0)=0$, and by convexity $(\Lambda_{u,1}^*)'(t)>0$ for all $(u,t)\in \Sigma\times \mathbb{R}$, we deduce the existence of $\eta>0$ such that for all $u\in \Sigma$ and all $t\leq 1$, 
\begin{equation*}
    \Lambda_{u,1}^*(t)>\eta t^2.
\end{equation*}
Hence, uniformly on $u\in \Sigma$, for $t\leq 1$,
\begin{equation*}
    \mathbb{P}_u\bigl(\vert S_1(n)-n\mu_1^u\vert> tn\bigr)\leq 2\exp(-n\eta t^2).
\end{equation*}
Doing the same for $S_2(n)$ yields the result.

Likewise, by strict convexity there exists $\eta'>0$ such that  
\begin{equation*}
    \Lambda_{u,1}^*(t)>\eta' t
\end{equation*}
for $t\geq 1$, independently of $u$. The second result is then deduced as the first one. 
\end{proof}

Finally, we will use a uniform bound on the Green function $\widetilde{G}$ on $\mathbb{Z}^2$, namely, there exists $C>0$ such that for all $x,y\in \mathbb{Z}^2$,
\begin{equation}\label{eq:uniform_bound_green}
\widetilde{G}(x,y)\leq C.
\end{equation}
The existence of such a bound is just a consequence of the transience of the walk $\{S(n)\}_{n\geq 1}$.

\section{Asymptotics of the Green function on the half-space}\label{Section:halfspace}

The most delicate part of the description of the Martin boundary is related to the behavior of the Green function along the half-axes $\mathbb{R}_{>0}e_1$ and $\mathbb{R}_{>0}e_2$. This situation is dealt by first giving an asymptotics of the Green function on the half-planes $\mathbb{H}_1$ and $\mathbb{H}_2$. 

For $x,y\in \mathbb{H}_i$, denote by 
\begin{equation}
\label{eq:def_exit_time_H}
    \tau^i_x=\inf\{n\geq 0: x+S(n)\not\in \mathbb{H}_i\}
\end{equation}
and introduce the Green kernels on the half-plane
\begin{equation*}
    \widehat{G}^i(x,y)=\sum_{n\geq 1}\mathbb{P}(x+S(n)=y,\tau^i_x>n) \quad\text{and}\quad \widehat{G}^i_u(x,y)=\sum_{n\geq 1}\mathbb{P}_u(x+S(n)=y,\tau^i_x>n).
\end{equation*}
The goal of this section is to prove the following result. For $z\in \mathbb{R}^2$, denote by $z_1$ (resp.\ $z_2$) its first (resp.\ second) coordinate.
\begin{prop}
\label{prop:asymptotic_Green_horizontal_final}
Let $\epsilon,\eta>0$ be small enough and $i\in\{1,2\}$. As $y$ goes to $\infty$ with $\vert y_i\vert =o(\vert y\vert^{1/2+\epsilon})$, then uniformly on $x\in\mathbb{H}_i$ with $\vert x\vert=o(\vert y\vert^{1/2-\eta})$,
\begin{equation*}
    \widehat{G}^i(x,y)\sim \frac{B(u(y)) x_{\bar{\imath}} }{\vert y\vert^{3/2}}V(y_{\bar{\imath}})e^{\langle u(y),x-y\rangle},
\end{equation*}
where $u(y)=\frac{y}{\vert y\vert}$, $u\mapsto B(u)$ is a continuous, positive function on $\Sigma$ and $V$ will be introduced in Proposition~\ref{prop:equiv_proba_horizon}.
\end{prop} 
We only prove it for $i=1$, the other case being similar. 
The proof of Proposition~\ref{prop:asymptotic_Green_horizontal_final} is decomposed in three cases, depending on the distance of the endpoint to the horizontal axis. To prove this result, we need two local limit theorems for the walk $x+S(n)$ conditioned to stay in $\mathbb{H}_1$. 

The first local limit theorem concerns the case where the distance of the endpoint $x+S(n)$ to the horizontal axis is similar to the fluctuation scale $\sqrt{n}$. The equivalence part of this result is given by \cite[Lem.~3.2]{DuRaTaWa-22} when $x_2=o(\sqrt{n})$ and by \cite[Lem.~3.1 (c)]{DuRaTaWa-22} when $x_2\geq c\sqrt{n}$ for some $c>0$, whereas the uniform bound is given by \cite[Lem.~2.3]{DuRaTaWa-22}.

\begin{prop}
\label{prop:equiv_proba_horizon}
There exists $\kappa>0$ such that for all $A>0$, as $n$ goes to infinity, uniformly on $x,y\in \mathbb{H}_1$ with $\vert x\vert=o(\sqrt{n})$ and $\vert x-n\mu^{e_1}\vert\leq A\sqrt{n}$,
\begin{equation*}
    \mathbb{P}_{e_1}(x+S(n)=y,\tau_x>n)\sim\frac{\kappa x_2V(y_2)}{2\pi \sqrt{\det(\Sigma^{e_1})}n^{2}}\exp\left(-\frac{\Vert C^{e_1} (y-n\mu^{e_1})\Vert^2}{2n}\right),
\end{equation*}
where $V$ is a positive harmonic function with respect to $\{-S_2(n)\}_{n\geq 1}$ killed at the boundary.
Moreover, there exists $C>0$ such that for all $x,y\in \mathbb{H}_1$ with $\vert x\vert=o(\sqrt{n})$,
\begin{equation*}
    \mathbb{P}_{e_1}(x+S_n=y,\tau_x>n)\leq C\frac{x_2V(y_2)}{n^2}\left(\exp\left(-\frac{\Vert C^u (y-n\mu^{e_1})\Vert^2}{2n}\right)+o\left(n^{-2}\right)\right).
\end{equation*}
\end{prop}
Remark that in the latter proposition, $C$ only depends on the behavior of $\frac{\vert x\vert}{\sqrt{n}}$. By \cite[Lem.~13 (a)]{DeWa-15}, we also have as $z$ goes to $\infty$
\begin{equation}
\label{eq:asymptotic_V_reverse}
   V(z)\sim z.
\end{equation}

The second local limit theorem deals with the case where the distance of the endpoint $x+S(n)$ to the horizontal axis is large compared to the fluctuation scale $\sqrt{n}$.
\begin{prop}
\label{prop:equiv_proba_micro}
For $\eta$ small enough, as $n$ goes to infinity, uniformly on $x,y\in \mathbb{H}_1$ with $\vert x\vert=o(\sqrt{n})$ and $u\in \Sigma$ with $n^{-1/2}=o(u_2)$ and $u_2\leq n^{-1/2+\eta}$,
\begin{equation*}
    \mathbb{P}_u(x+S(n)=y,\tau_x>n)=\frac{c(u)x_2\mu_2^u}{\Sigma^u_{11}2\pi n\sqrt{\det(\Sigma^u)}}\exp\left(-\frac{\Vert C^u (y-n\mu^u)\Vert^2}{2n}\right)+o\left(\frac{x_2\mu^u_2}{n}\right),
\end{equation*}
with $o(\cdot)$ uniform on $u$ and $x,y\in \mathbb{H}_1$. Moreover, there exists $C>0$ such that for all $x,y\in \mathbb{H}_1$ with $\vert x\vert=o(\sqrt{n})$,
\begin{equation*}
    \mathbb{P}_u(x+S_n=y,\tau_x>n)\leq Cx_2\frac{\mu_2^u}{t}\left(\exp\left(-\frac{\Vert C^u (y-n\mu^u)\Vert^2}{2n}\right)+o\left(n^{-2}\right)\right).
\end{equation*}
\end{prop}
The proof of the latter proposition is very similar to the one of Proposition \ref{prop:equiv_proba_horizon} and is given in Appendix \ref{Appendix:llt}.

\subsection{Asymptotics at mesoscopic distance}
Since the support of the random walk $S$ is $\{(i,j)\in\mathbb{Z}^2: i,j\geq -1\}$ by our assumptions \ref{it:small_neg}, we have almost surely $S_2(1)\geq -1$. Such a random variable is called skip-free in the literature \cite{BPR-10} and has interesting properties allowing to explicitly compute the survival probability. 

For $z>0$, denote by $\tau'_z=\inf\{n\geq 0: z+S_2(n)=0\}$, and remark that $\tau^1_x=\tau'_{x_2}$ for $x\in \mathbb{H}_1$, see \eqref{eq:def_exit_time_H}. Then, the skip-free assumption yields that
\begin{enumerate}[label=($\textrm{SF}\arabic{*}$),ref=$\textrm{SF}\arabic{*}$]
   \item\label{eq:survival_probability_small_mu_1}if $\mu^u_2>0$, $\mathbb{P}_u(\tau_{x_2}'<\infty)=\mathbb{P}_u(\tau_{1}'<\infty)^{x_2}=c_u^{x_2}$, where $c_u\in (0,1)$ is the unique real in $(0,1)$ such that $\mathbb{E}_u(c^{S_2(1)})=1$;
   \item\label{eq:survival_probability_small_mu_2}as $\mu^u_2$ goes to $0$, one has $1-c_u\sim 2\frac{\mu^u_2}{\Sigma^u_{22}}$.
\end{enumerate}

We have the following first asymptotic result for the Green function on $\mathbb{H}_1$, when the endpoint remains far from the horizontal axis.
\begin{lem}
\label{lem::half_green_function_distant}
Let $\epsilon>0$. For any $x\in \mathbb{H}_1$, as $t\to\infty$, uniformly on $u\in \Sigma$ with $\mu^u_2\geq t^{-1/2+\epsilon}$,
\begin{equation*}
    \sqrt{t}\widehat{G}_u(x,\lfloor t\mu^u\rfloor)\sim A(u)\mathbb{P}_u(\tau^1=\infty)\sim B(u)x_2u_2.
\end{equation*}
\end{lem}

\begin{proof}
For $x\in \mathbb{H}_1$, we have $\tau^1_x=\inf\{n\geq 0: x_2+S_2(n)\leq 0\}$, see \eqref{eq:def_exit_time_H}. Set $t_u=\lfloor t\mu^u\rfloor$. Then, for $n_t=\lfloor t^{1-\epsilon}\rfloor$,
\begin{align*}
\widehat{G}_u(x,t_u)&=\sum_{n=1}^{n_t}\mathbb{P}_u\bigl(x+S(n)=t_u,\tau^1_x>n\bigr)+\mathbb{E}_u\bigl(G_u(x+S({n_t}),t_u),\tau_{x_2}'>n_t\bigr)\\
&=\sum_{n=0}^{n_t}\mathbb{P}_u(x+S(n)=t_u,\tau_{x_2}'>n)+\mathbb{E}_u\bigl(\widetilde{G}_u(x+S({n_t)},t_u),\tau_{x_2}'>n_t\bigr)\\
&\hspace{5.59cm}-\mathbb{E}_u\bigl(\widetilde{G}_u(x+S({\tau_x}),t_u),n_t<\tau_{x_2}'<\infty\bigr)\\
&=:S_1+S_2+S_3.
\end{align*}
By Lemma \ref{lem:large_deviation_uniform},
\begin{align*}
S_1&\leq \sum_{n=1}^{n_t}\mathbb{P}_u\left(\vert S(n) -n\mu^u\vert \geq\vert \lfloor t\mu^u\rfloor-x-n\mu^u\vert\right)\leq \sum_{n=1}^{n_t}\mathbb{P}_u\left(\vert S(n) -n\mu^u\vert >\frac{\vert t\mu^u \vert}{2n}n\right)\\
&\leq 4\sum_{n=1}^{n_t}\exp\left(-\frac{\eta \vert t\mu^u \vert^2}{(2n)^2}\right)\leq n_t\exp\left(-\frac{\eta \vert t\mu^u \vert^2}{(2n_t)^2}\right).
\end{align*}
Hence, since $n_t\leq t^{1-\epsilon}$,
\begin{equation*}
   S_1=O\bigl(\exp(-ct^{2\epsilon})\bigr)
\end{equation*}
for some $c>0$. Then, set $\alpha=t^{1/2-\epsilon/4}$ and write 
\begin{multline*}
S_2=\mathbb{E}_u\Bigl(\widetilde{G}_u(x+S(n_t),t_u),\vert S(n_t)-n_t \mu^u\vert \leq \alpha, \tau^1_x>n_t\Bigr)\\
+\mathbb{E}_u\Bigl(\widetilde{G}_u(x+S({n_t}),t_u),\vert S({n_t})-n_t \mu^u\vert > \alpha, \tau^1_x>n_t\Bigr)=:S_{21}+S_{22}.
\end{multline*}
First, by \eqref{eq:uniform_bound_green} and Lemma \ref{lem:large_deviation_uniform},
\begin{equation}
\label{eq:boundproba_alpha}
S_{22}\leq C\mathbb{P}_u(\vert S(n_t)-n_t \mu^u\vert > \alpha)\leq 4C\exp(-\eta (\alpha/n_t)^2n_t)=O(\exp(-\eta t^{\epsilon/2})).
\end{equation}
Then, on the event $\{\vert S(n_t)-n_t\mu^{u}\vert\leq \alpha\}$, as $t$ goes to infinity $\frac{\vert x+S(n_t)-n_t\mu\vert}{t}=o(t^{1/2-\epsilon})$. Hence, by Theorem~\ref{thm:asymptotic_green_Ney_spiter}, uniformly on $x+S(n_t)$ on the event $\{\vert S(n_t)\vert \leq \alpha\}$,
\begin{equation*}
    \sqrt{t}\widetilde{G}_{u}(x+S(n_t),t_u)\xrightarrow[] {t\rightarrow \infty}A(u).
\end{equation*}
Hence, $S_{21}\sim_{t\rightarrow \infty}\frac{A(u)}{\sqrt{t}}\mathbb{P}_u\bigl(\vert S(n_t)-n_t \mu^u\vert \leq \alpha, \tau^1_x>n_t\bigr)$ as $t$ goes to $\infty$. Split the latter probability as
\begin{multline*}
\mathbb{P}_u\bigl(\vert S(n_t)-n_t \mu^u\vert \leq \alpha, \tau^1_x>n_t\bigr)=\mathbb{P}_u\bigl(\vert S(n_t)-n_t \mu^u\vert \leq \alpha, n_t<\tau^1_x<\infty\bigr)\\+\mathbb{P}_u\bigl(\tau^1_x=\infty\bigr)-\mathbb{P}_u\bigl(\vert S(n_t)-n_t \mu^u\vert > \alpha, \tau^1_x=\infty\bigr).
\end{multline*}
First, 
\begin{equation*}
    \mathbb{P}_u\left(\vert S(n_t)-n_t \mu^u\vert > \alpha, \tau^1_x=\infty\right)\leq \mathbb{P}_u\left(\vert S(n_t)-n_t \mu^u\vert >\alpha\right)=O(\exp(-\eta t^{\epsilon/2}))
\end{equation*}
as in \eqref{eq:boundproba_alpha}. Then, on the event $\{\vert S(n_t)-n_t \mu^u\vert \leq \alpha, \tau_x>n_t\}$, $v:=x+S(n_t)$ satisfies 
\begin{equation*}
    v_2\geq x_2+n_t\mu^u_2-\alpha\geq ct^{1-\epsilon}\mu_2^u -t^{1/2-\epsilon/4}\geq ct^{1/2}
\end{equation*}
for some constant $c>0$, thanks to the condition $\mu_2^u\geq t^{-1/2+\epsilon}$. Hence, by \eqref{eq:survival_probability_small_mu_1} and \eqref{eq:survival_probability_small_mu_2}, on the event $\{\vert S(n_t)-n_t \mu^u\vert \leq \alpha, \tau^1_x>n_t\}$, we have 
\begin{equation}
\label{eq:small_exit_proba}
\mathbb{P}\bigl(\tau^1_{x+S(n_t)}<\infty\vert \,\vert S(n_t)-n_t \mu^u\vert \leq \alpha, \tau^1_x>n_t\bigr)\leq \left(1-c_1\frac{\mu_2^u}{\Sigma^u_{22}}\right)^{c_2t^{1/2}}\leq \exp(-ct^{\epsilon})
\end{equation}
for some constant $c>0$, where we used again that $\mu_2^u\geq t^{-1/2+\epsilon}$ in the last inequality. This implies that 
$\mathbb{P}_u\bigl(\vert S(n_t)-n_t \mu^u\vert \leq \alpha, n_t<\tau^1_x<\infty\bigr)=O\bigl(\exp(-ct^{\epsilon})\bigr)$.
Finally, by \eqref{eq:survival_probability_small_mu_1} and \eqref{eq:survival_probability_small_mu_2} with $\mu_2^u=o(1)$, $\mu^u_2\geq t^{-1/2+\epsilon}$,
\begin{equation*}
\mathbb{P}_u(\tau_{x_2}'=\infty)=1-c_u^{x_2}\geq 1-\bigl(1-ct^{-1/2+\epsilon}\bigr)^{x_2}\geq c'x_2t^{-1/2+\epsilon}
\end{equation*}
for some constants $c,c'$ depending on $x$, and we finally have $S_{21}\sim_{t\rightarrow \infty} \frac{A(u)}{\sqrt{t}}\mathbb{P}_u(\tau^1_x=\infty)$ and then 
\begin{equation*}
    S_2\sim_{t\rightarrow \infty} \frac{A(u)}{\sqrt{t}}\mathbb{P}_u(\tau^1_x=\infty).
\end{equation*}
Note that the latter term will be the main contribution to the asymptotics of $\widehat{G}_u(x,t_u)$.
Finally, by \eqref{eq:uniform_bound_green},
\begin{align*}
S_{3}&=\mathbb{E}_u\bigl(\widetilde{G}_u(x+S({\tau_x}),t_u),n_t<\tau^1_x<\infty\bigr)\\
&\leq C\mathbb{P}_u\bigl(n_t<\tau^1_x<\infty\bigr)\\
&\leq C\left(\mathbb{P}_u\bigl(\vert S(n_t)-n_t\mu^u\vert\leq \alpha,n_t<\tau^1_x<\infty\bigr)+\mathbb{P}_u\bigl(\vert S(n_t)-n_t\mu\vert> \alpha\bigr)\right).
\end{align*}
By \eqref{eq:small_exit_proba}, the first term is bounded by $\exp(-ct^\epsilon)$ and by Lemma \ref{lem:large_deviation_uniform} the second term is bounded by $\exp(-\eta t^{\epsilon/2})$. Hence,
\begin{equation*}
S_3=o\bigl(\exp(-\eta t^{\epsilon/2})\bigr).
\end{equation*}
Putting together the results on $S_1$, $S_2$ and $S_3$ yields then
\begin{equation*}
    \widehat{G}_u(x,t_u)\sim_{t\rightarrow\infty}\frac{A(u)}{\sqrt{t}}\mathbb{P}_u(\tau^1_x=\infty)
\end{equation*}
uniformly on $u$ such that $u_2\geq t^{-1/2+\epsilon}$. 
The second equivalence in Lemma~\ref{lem::half_green_function_distant} is a consequence of \eqref{eq:survival_probability_small_mu_1} and \eqref{eq:survival_probability_small_mu_2}.
\end{proof}

\subsection{Asymptotics of the Green function at microscopic distance}
The proof essentially follows the original proof of Ney and Spitzer \cite{NeSpi-66}.
\begin{prop}
\label{prop:asymptotic_Green_meso}
Let $(\theta_t)_{t\geq 0}$ a positive function going to zero at $\infty$ and $\epsilon >0$.  For $\eta>0$ small enough, as $t$ goes to $\infty$, uniformly on $u\in \Sigma$ with $\frac{t^{-1/2}}{\theta_t}\leq u_2\leq t^{-1/2+\eta}$ and $x\in\mathbb{H}_1$ with $\vert x\vert=o(t^{1/2-\epsilon})$,
\begin{equation*}
    \sqrt{t}\widehat{G}_u(x,\lfloor t\mu^u\rfloor)\sim B(u)x_2u_2.
\end{equation*}
\end{prop}
\begin{proof}
Write $y=\lfloor t\mu^u\rfloor$ and split $\widehat{G}_u(x,\lfloor t\mu^u\rfloor)$ as
\begin{align*}
\widehat{G}_u(x,\lfloor t\mu^u\rfloor)&=\sum_{n=1}^{\lfloor t-A\sqrt{t}\rfloor-1}\mathbb{P}_u\bigl(x+S(n)=y,\tau^1_x>n\bigr)+\sum_{n=\lfloor t-A\sqrt{t}\rfloor}^{{\lfloor t+A\sqrt{t}\rfloor}}\mathbb{P}_u\bigl(x+S(n)=y,\tau^1_x>n\bigr)\\
&\hspace{6.04cm}+\sum_{n=\lfloor t+A\sqrt{t}\rfloor+1}^{\infty}\mathbb{P}_u\bigl(x+S(n)=y,\tau^1_x>n\bigr)\\
&:=S_1+S_2+S_3.
\end{align*}
First, by the assumption on $u_2$, $n^{-1/2}=o(u_2)$ and $u_2\leq Cn^{-1/2+\eta}$ for $\lfloor t-A\sqrt{t}\rfloor\leq n\leq \lfloor t+A\sqrt{t}\rfloor$, with $o(\cdot)$ and $C$ only depending on $A$ and $\theta_t$. Hence, by Proposition \ref{prop:equiv_proba_micro}, uniformly for $x,y\in\mathbb{H}_1$ with $\vert x\vert=o(\sqrt{n})$ and $\lfloor t-A\sqrt{t}\rfloor\leq n\leq \lfloor t+A\sqrt{t}\rfloor$ and $u\in \Sigma$ with $\frac{t^{-1/2}}{\theta_t}\leq u_2\leq t^{-1/2+\epsilon}$,
\begin{equation*}
    \mathbb{P}_u\bigl(x+S(n)=y,\tau_x>n\bigr)=\frac{c(u)x_2\mu_2^u}{\Sigma^u_{11}2\pi n\sqrt{\det(\Sigma^u)}}\exp\left(-\frac{\Vert C^u (y-n\mu^u)\Vert^2}{2n}\right)+o\left(\frac{x_2\mu^u_2}{n}\right).
\end{equation*}
Applying the latter to $S_2$ yields
\begin{align*}
\sqrt{t}S_2&=\sqrt{t}\left(\sum_{n=\lfloor t-A\sqrt{t}\rfloor}^{{\lfloor t+A\sqrt{t}\rfloor}}\frac{c(u)x_2\mu_2^u}{\Sigma^u_{11}2\pi n\sqrt{\det(\Sigma^u)}}\exp\left(-\frac{\Vert C^u (y-n\mu^u)\Vert^2}{2n}\right)+o\left(\frac{x_2\mu^u_2}{n}\right)\right)\\
&=\sqrt{t}\frac{c_ux_2\mu_2^u}{\Sigma^u_{11}}\left(\sum_{n=\lfloor t-A\sqrt{t}\rfloor}^{{\lfloor t+A\sqrt{t}\rfloor}}\frac{1}{2\pi n\sqrt{\det(\Sigma^u)}}\exp\left(-\frac{\Vert C^u (y-n\mu^u)\Vert^2}{2n}\right)\right)+x_2\mu^u_2 o\left(1\right).
\end{align*}
By \cite[Eq.~(2.20)]{NeSpi-66},
\begin{equation*}
    \sqrt{t}\sum_{n=\lfloor t-A\sqrt{t}\rfloor}^{{\lfloor t+A\sqrt{t}\rfloor}}\frac{1}{2\pi n\sqrt{\det(\Sigma^u)}}\exp\left(-\frac{\Vert C^u (y-n\mu^u)\Vert^2}{2n}\right)\xrightarrow[]{t\rightarrow\infty}\frac{1}{\sqrt{2\pi\det(\Sigma^u)}\Vert C^u\mu^u\Vert}+\epsilon(A),
\end{equation*}
uniformly in $u\in \Sigma$, with $\epsilon(A)\xrightarrow[]{A\rightarrow\infty}0$ and $\vert x\vert=t^{1/2-\epsilon}$. Hence, we also have
\begin{equation}\label{eq:term_S2_meso}
\sqrt{t}S_2\sim_{t\rightarrow\infty}\frac{c_ux_2\mu_2^u}{\Sigma^u_{11}\sqrt{2\pi\det(\Sigma^u)}\Vert C^u\mu^u\Vert}(1+\epsilon(A)).
\end{equation}
Then, for $0<\delta <1$, split $S_1$ as 
\begin{equation*}
    S_1=\sum_{n=1}^{\lfloor\delta t\rfloor}\mathbb{P}_u\bigl(x+S(n)=y,\tau^1_x>n\bigr)+\sum_{n=\lfloor\delta t\rfloor+1}^{\lfloor t-A\sqrt{t}\rfloor-1}\mathbb{P}_u\bigl(x+S(n)=y,\tau^1_x>n\bigr)=:S_{11}+S_{12}.
\end{equation*}
By Lemma \ref{lem:large_deviation_uniform}, the first term is bounded above as 
\begin{equation*}
    \frac{\sqrt{t}}{x_2\mu_2^u}S_{11}\leq \sqrt{t}\sum_{n=1}^{\lfloor\delta t\rfloor}\mathbb{P}_u\bigl(x+S(n)=y\bigr)\leq Ct^{3/2}\exp\bigl(-c(1-\delta)t\bigr)\xrightarrow{t\rightarrow\infty} 0
\end{equation*}
uniformly in $\vert x\vert<t^{1/2-\epsilon}$, $A>0$ and $u\in \Sigma$ (see \cite[Eq.~(2.22)]{NeSpi-66}). By Proposition~\ref{prop:equiv_proba_micro} with $r=3$, the second term is bounded as 
\begin{align*}
\frac{\sqrt{t}}{x_2\mu_2^u}S_{12}\leq \sqrt{t}\sum_{n=\lfloor\delta t\rfloor+1}^{\lfloor t-A\sqrt{t}\rfloor-1}\left(\frac{Cx_2\mu_2}{n}\exp\left(-c^u \Vert y-n\mu^u\Vert^2\right)+o\left(n^{-2}\right)\right)\\
\leq \sqrt{t}\sum_{n=\lfloor\delta t\rfloor+1}^{\lfloor t-A\sqrt{t}\rfloor-1}\left(\frac{Cx_2\mu_2}{n}\exp\left(-c^u \Vert y-n\mu^u\Vert^2\right)\right)+o\left(\frac{1}{t}\right).
\end{align*}
By \cite[Eq.~(2.26)]{NeSpi-66},
\begin{equation*}
    \limsup_{t\rightarrow \infty}\sqrt{t}\sum_{n=\lfloor\delta t\rfloor+1}^{\lfloor t-A\sqrt{t}\rfloor-1}\frac{1}{n}\exp\left(-c^u \Vert y-n\mu^u\Vert^2\right)= \epsilon'(A)
\end{equation*}
uniformly on $u\in\Sigma$, with $\epsilon'(A)$ going to $0$ when $A$ going to $\infty$. Hence,
\begin{equation*}
    \limsup_{t\rightarrow \infty}\frac{\sqrt{t}}{x_2\mu_2^u}\sum_{n=\lfloor\delta t\rfloor+1}^{\lfloor t-A\sqrt{t}\rfloor-1}\mathbb{P}_u(x+S(n)=y,\tau^1_x>n)= \epsilon'(A).
\end{equation*}
Finally,
\begin{equation}\label{eq:term_S1_meso}
\limsup_{t\rightarrow \infty}\frac{\sqrt{t}}{x_2\mu_2^u}S_1= \epsilon'(A),
\end{equation}
uniformly on $u\in \Sigma$ and $x\in \mathbb{H}_1$ with $\vert x\vert<t^{1/2-\epsilon}$. One proves similarly that 
\begin{equation}\label{eq:term_S3_meso}
\limsup_{t\rightarrow \infty}\frac{\sqrt{t}}{x_2\mu_2^u}S_3= \epsilon'(A),
\end{equation}
uniformly on $u\in \Sigma$ and $x\in \mathbb{H}_1$ with $\vert x\vert<t^{1/2-\epsilon}$.  Putting \eqref{eq:term_S2_meso}, \eqref{eq:term_S1_meso} and \eqref{eq:term_S3_meso} together, we get
\begin{align*}
1-\epsilon(A)-2\epsilon'(A)<&\liminf \frac{\sqrt{t}\widehat{G}_u(x,\lfloor t\mu^u\rfloor)\Sigma^u_{11}\sqrt{2\pi\det(\Sigma^u)}\Vert C^u\mu^u\Vert}{c_ux_2\mu_2^u}\\
\leq& \limsup \frac{\sqrt{t}\widehat{G}_u(x,\lfloor t\mu^u\rfloor)\Sigma^u_{11}\sqrt{2\pi\det(\Sigma^u)}\Vert C^u\mu^u\Vert}{c_ux_2\mu_2^u}<1+\epsilon(A)+2\epsilon'(A)
\end{align*}
for all $A>0$. Letting $A$ go to $\infty$, we thus get
\begin{equation*}
    \lim_{t\rightarrow \infty}\frac{\sqrt{t}\widehat{G}_u(x,\lfloor t\mu^u\rfloor)}{x_2\mu_2^u}=\frac{c_u}{\Sigma^u_{11}\sqrt{2\pi\det(\Sigma^u)}\Vert C^u\mu^u\Vert}
\end{equation*}
uniformly on $u\in \Sigma$ and $x\in \mathbb{H}_1$ with $\vert x\vert=o(t^{1/2-\epsilon})$.
\end{proof}
We give a similar asymptotic result along the horizontal axis.
\begin{prop}
\label{prop:asymptotic_Green_micro}
Let $B>0$. Uniformly on $x\in\mathbb{H}_1$ with $\vert x\vert=o(t^{1/2-\epsilon})$, $y=t\mu^{e_1}+t'e_2$ with $0\leq t'\leq B\sqrt{t}$,
\begin{equation*}
    \widehat{G}_{e_1}(x,y)\sim  \frac{1}{\vert y\vert^{3/2}}B(e_1)V(y'_2)x_2e^{-\langle\phi(u(y))-\phi(e_1), y\rangle},
\end{equation*}
with $V$ introduced in \eqref{eq:asymptotic_V_reverse} and $B(e_1)>0$.
\end{prop}
\begin{proof}
The proof follows the lines of the one of Proposition \ref{prop:asymptotic_Green_meso}, using Proposition \ref{prop:equiv_proba_horizon} instead of Proposition \ref{prop:equiv_proba_micro}. The only different step is the computation of 
\begin{equation*}
    \sqrt{t}\left(\sum_{n=\lfloor t-A\sqrt{t}\rfloor}^{{\lfloor t+A\sqrt{t}\rfloor}}\frac{1}{2\pi n\sqrt{\det(\Sigma^u)}}\exp\left(-\frac{\Vert C^u (\lfloor t\mu^u\rfloor-n\mu^u)\Vert^2}{2n}\right)\right)
\end{equation*}
which has to be replaced by 
\begin{equation*}
    \sqrt{t}\left(\sum_{n=\lfloor t-A\sqrt{t}\rfloor}^{{\lfloor t+A\sqrt{t}\rfloor}}\frac{1}{2\pi n^2\sqrt{\det(\Sigma^{e_1})}}\exp\left(-\frac{\Vert C^{e_1} (\lfloor t\mu^{e_1}+t'e_2\rfloor-n\mu^{e_1})\Vert^2}{2n}\right)\right):=S_A.
\end{equation*}
As $t$ goes to infinity and for $\lfloor t-A\sqrt{t}\rfloor\leq n\leq \lfloor t+A\sqrt{t}\rfloor$, $n\sim t$, and thus 
\begin{equation*}
    S_A\sim \frac{1}{\sqrt{t}}\left(\sum_{n=\lfloor t-A\sqrt{t}\rfloor}^{{\lfloor t+A\sqrt{t}\rfloor}}\frac{1}{2\pi n\sqrt{\det(\Sigma^{e_1})}}\exp\left(-\frac{\Vert C^{e_1} (\lfloor t\mu^{e_1}+t'e_2\rfloor-n\mu^{e_1})\Vert^2}{2n}\right)\right).
\end{equation*}
By \cite{NeSpi-66}, uniformly on $t'\leq B\sqrt{t}$,
\begin{multline*}
\lim_{t\rightarrow \infty}\sum_{n=\lfloor t-A\sqrt{t}\rfloor}^{{\lfloor t+A\sqrt{t}\rfloor}}\frac{1}{2\pi n\sqrt{\det(\Sigma^{e_1})}}\exp\left(-\frac{\Vert C^{e_1} (\lfloor t\mu^{e_1}+t'e_2\rfloor-n\mu^{e_1})\Vert^2}{2n}\right)\\
\hspace{3cm}\sim A(u(y))\exp(-\langle \phi(u(y))-\phi(e_1),y\rangle)+\epsilon(A),
\end{multline*}
with $\epsilon(A)$ going to $0$ as $A$ goes to infinity. Hence,
\begin{equation*}
    \lim_{A\rightarrow \infty} \lim_{t\rightarrow \infty}\sqrt{t}S_A=A(u(y))\exp(-\langle \phi(u(y))-\phi(e_1),y\rangle).
\end{equation*}
Concluding the proof as the one of Proposition \ref{prop:asymptotic_Green_meso}, we get
\begin{equation*}
    \widehat{G}_{e_2}(x,y)\sim_{t\rightarrow \infty} \frac{B(e_2)}{\vert y\vert^{3/2}}V(y_2)x_2e^{-\langle \phi(u(y))-\phi(e_1),y\rangle},
\end{equation*}
with $B(e_2)=\kappa \widetilde{A}(e_2)$.
\end{proof}

\begin{proof}[Proof of Proposition \ref{prop:asymptotic_Green_horizontal_final}]
When $\sqrt{\vert y\vert}=o(y_2)$, the statement of Proposition \ref{prop:asymptotic_Green_horizontal_final} is obtained by putting together Lemma \ref{lem::half_green_function_distant} and Proposition \ref{prop:asymptotic_Green_meso} with $u=u(y)$ and $t=\frac{\vert y\vert}{\vert \mu^{u(y)}\vert}$, using that $\mu_2^{u(y)}\sim \frac{y_2}{\vert y\vert}$ and $V(z)\sim z$ as $z$ goes to infinity.
When $y_2\leq B\vert y_1\vert^{1/2}$ for some $B>0$, use Proposition \ref{prop:asymptotic_Green_micro} with $t=\frac{y_1}{\vert \mu^{e_1}\vert}\sim \frac{\vert y\vert}{\vert \mu^{e_1}\vert}$ to get 
\begin{equation*}
    \widehat{G}(x,y)\sim \frac{B(e_1)x_2V(y_2)}{\vert y\vert^{3/2}}e^{\langle\phi(e_1),y\rangle-\langle\phi(u(y)), y\rangle}
\end{equation*}
with $B(e_1)>0$. The result is deduced provided we prove that 
\begin{equation*}
    \langle\phi(e_1),y\rangle=\langle \phi(u(y)), x\rangle +o(1)
\end{equation*}
uniformly on $y=o(\vert x\vert^{1/2-\epsilon})$. As $y$ goes to infinity, $u(y)=e_1+\frac{y_2}{\vert y\vert}e_2+o\bigl(\frac{y_2}{\vert y\vert}\bigr)$. Since $\phi$ is $\mathcal{C}^2$, $K(\phi(u))=1$ and $\nabla K\circ \phi(u)\in \mathbb{R}u$. For $\delta>0$ small enough with $u+\delta\in \Sigma$, we have $\phi(u+\delta)=\phi(u)+u^{\perp}h_u(\delta)+O(\vert \delta\vert^2)$ for some linear functional $h_u$. Hence,
\begin{equation*}
    \phi(u(y))=\phi(e_1)+\alpha\frac{y_2}{\vert y\vert}e_2+O\left(\frac{y_2}{\vert y\vert}\right)^2.
\end{equation*}
Therefore, for $x=o(\vert y\vert^{1/2-\epsilon})$,
\begin{equation*}
    \langle \phi(u(y)), x\rangle =\langle\phi(e_1),y\rangle+o\left(\frac{y_2}{\vert y\vert^{1/2+\epsilon}}\right)= \langle\phi(e_1),y\rangle+o(1).
\end{equation*}
The result is then deduced.
\end{proof}

\section{Identification of the Martin Boundary}\label{Section:proof_main_results}
We can now describe the Martin boundary $\partial_{\mathbb{Z}_{>0}^2}^MS$. We first prove Theorem \ref{thm:asymptotic_inside_green}, and then deduce from the latter the following description:
\begin{cor}
\label{cor:MB}
The Martin boundary $\partial_M^{\mathbb{Z}_{>0}^2}S$ of $\{S(n)\}_{n\geq 0}$ killed outside of the quarter plane $\mathbb{Z}_{>0}^2$ is homeomorphic to $\overline{\mathcal{G}_0}$ through the map
\begin{equation*}
    (a_0,b_0)\in \overline{\mathcal{G}_0}\mapsto \frac{h_{(a_0,b_0)}}{h_{(a_0,b_0)}(1,1)}\in\partial_M^{\mathbb{Z}_{>0}^2}S.
\end{equation*}
More precisely, for $(a_0,b_0)\in\mathcal{G}_0$, writing $u=\frac{\nabla K(a_0,b_0)}{\Vert \nabla K(a_0,b_0)\Vert}$,
\begin{equation*}
   h_{(a_0,b_0)}(x)=e^{\langle (a_0,b_0), x\rangle}\mathbb{P}_{u}(\tau_x=\infty)=e^{\langle (a_0,b_0), x\rangle}-\mathbb{E}\left(e^{\langle(a_0,b_0), x+S(\tau_x)\rangle},\tau_x<\infty\right)
\end{equation*}
and if $(a_0,b_0)\in\partial \mathcal{G}_0$, with $e_i=\frac{\nabla K(a_0,b_0)}{\Vert \nabla K(a_0,b_0)\Vert}$, $i\in\{1,2\}$,
\begin{equation*}
h_{(a_0,b_0)}(x)=x_{\bar{\imath}}e^{\langle (a_0,b_0),x\rangle}-\mathbb{E}\left(e^{\langle(a_0,b_0),x+S(\tau_x)\rangle}(x_{\bar{\imath}}+S_{\bar{\imath}}(\tau_x))\right).
\end{equation*}
\end{cor}

\begin{proof}[Proof of Theorem \ref{thm:asymptotic_inside_green}]
We only prove the first statement and sketch the proof of the second one, which is similar to the first proof.

Let $V(u_0)$ be a neighborhood of $u_0\in\Sigma$ such that $d(V(u_0),\partial \Sigma)>0$. Set $t_n=\frac{\vert y(n)\vert}{\vert \mu^{u(y(n))}\vert}$ and $u_n=\frac{y(n)}{\vert y(n)\vert}$. Then, for $m_n\geq 1$ to adjust later,
\begin{align}
G_{u_n}(x,y(n))&=\sum_{m=1}^{m_n}\mathbb{P}_{u_n}(x+S(m)=y(n),\tau_x>m)+\mathbb{E}_{u_n}\left(G_{u_n}(x+S({m_n}),y(n)),\tau_x>m_n\right)\nonumber\\
&=\sum_{m=0}^{m_n}\mathbb{P}_{u_n}(x+S(m)=y(n),\tau_x>m)+\mathbb{E}_{u_n}\left(\widetilde{G}_u(x+S({m_n}),y(n)),\tau_x>m_n\right)\nonumber\\
&\hspace{5cm}-\mathbb{E}_{u_n}\left(\widetilde{G}_u(x+S({\tau_x}),y(n)),m_n<\tau_x<\infty\right)\label{eq:renewal_decomposition}\\
&=:S_1+S_2+S_3.\nonumber
\end{align}
By the second part of Lemma \ref{lem:large_deviation_uniform}, as long as $m_n<t_n/3$ and for $\vert y(n)\vert$ large enough
\begin{align*}
S_1&\leq \sum_{n=1}^{m_n}\mathbb{P}_{u_n}\left(\vert S(m) -m\mu^{u(n)}\vert \geq\left\vert \lfloor t_n\mu^{u(n)}\rfloor-x-m\mu^{u(n)}\right\vert\right)\\
&\leq \sum_{n=1}^{m_n}\mathbb{P}_{u_n}\left(\vert S(m) -m\mu^{u(n)}\vert >\frac{\vert t_n\mu^{u(n)} \vert}{2m}m\right)\\
&\leq \sum_{n=1}^{m_n}\exp\left(-\frac{\eta' \vert t\mu^{u(n)} \vert}{2m}\right)\leq m_n\exp\left(-\frac{\eta' \vert t_n\mu^{u(n)} \vert}{2m_n}\right).
\end{align*}
Hence, as $m_n=o(t_n^{1/4})$,
\begin{equation}\label{eq:green_S_1}
S_1=o(\exp(-\sqrt{t_n})).
\end{equation}
Then, write 
\begin{align*}
S_2=&\mathbb{E}_{u_n}\left(\widetilde{G}_{u_n}(x+S({m_n}),y(n)),\vert S({m_n})-m_n \mu^{u(n)}\vert \leq 2m_n, \tau_x>m_n\right)\\
&\hspace{3cm}+\mathbb{E}_{u_n}\left(\widetilde{G}_{u_n}(x+S({m_n}),y(n)),\vert S({m_n})-m_n \mu^{u(n)}\vert > 2m_n, \tau_x>m_n\right)\\
=&S_{21}+S_{22}.
\end{align*}
First, by \eqref{eq:uniform_bound_green} and Lemma \ref{lem:large_deviation_uniform},
\begin{equation*}
S_{22}\leq C\mathbb{P}_{u_n}(\vert S({m_n})-m_n \mu^{u(n)}\vert > 2m_n)\leq C\exp(-2\eta' m_n).
\end{equation*}
Then, set $\widetilde{u}=\frac{y(n)-(x+S({m_n}))}{\vert y(n)-(x+S({m_n}))\vert}\in \mathbb{S}^1$. There exists $t_0$ such that on the event $\{\vert S({m_n})-n\mu^{u}\vert\leq 2m_n\}$, for $t_n\geq t_0$, $\widetilde{u}\in V(u_0)$. Then, by Theorem \ref{thm:asymptotic_green_Ney_spiter}, as $t_n$ goes to infinity with $m_n=o(t_n^{1/4})$, uniformly on $x+S({m_n})$ on the event $\{\vert S({m_n})\vert \leq 2m_n\}$ and using the fact that $u_n\rightarrow u_0$,
\begin{equation*}
    \sqrt{t}\widetilde{G}_{u_n}(x+S({m_n}),y(n))\xrightarrow[] {t\rightarrow \infty}A(u_0).
\end{equation*}
Hence,
\begin{multline*}
    S_{21}\sim_{n\rightarrow \infty}\mathbb{E}_{u_n}\left(\frac{A(u_0)}{\sqrt{t_n}}\mathbf{1}_{\vert S({m_n})-m_n \mu^{u(n)}\vert \leq 2m_n, \tau_x>m_n}\right)\\=\frac{A(u_0)}{\sqrt{t_n}}\mathbb{P}_{u_n}\left(\vert S({m_n})-m_n \mu^{u(n)}\vert \leq 2m_n, \tau_x>m_n\right).
\end{multline*}
Since 
\begin{multline*}
\mathbb{P}_{u_n}\left(\vert S({m_n})-m_n \mu^{u(n)}\vert \leq 2m_n, \tau_x>m_n\right)\\=\mathbb{P}_{u_n}\left( \tau_x>m_n\right)-\mathbb{P}_{u_n}\left(\vert S({m_n})-m_n \mu^{u(n)}\vert >2m_n, \tau_x>m_n\right),
\end{multline*}
with 
\begin{equation*}
\mathbb{P}_{u_n}( \tau_x>m_n)\sim_{n\rightarrow\infty}\mathbb{P}_{u_n}(\tau_x=\infty)\sim_{n\rightarrow\infty}\mathbb{P}_{u_0}(\tau_x=\infty)
\end{equation*}
and $\mathbb{P}_{u_n}(\vert S({m_n})-m_n \mu^{u(n)}\vert >2m_n, \tau_x>m_n)\leq \exp(-2\eta' m_n)$ by Lemma \ref{lem:large_deviation_uniform},
\begin{equation*}
    S_{21}\sim_{n\rightarrow \infty}\frac{A(u)\mathbb{P}_{u_0}( \tau_x=\infty)}{\sqrt{t_n}},
\end{equation*}
and as long as $m_n=o(t_n^{1/4})$, and $\log(t_n)=o(m_n)$,
\begin{equation}
\label{eq:green_S2}
S_2\sim_{t\rightarrow \infty}\frac{A(u_0)\mathbb{P}_{u_0}( \tau_x=\infty)}{\sqrt{t}}.
\end{equation}
As in the proof of Lemma \ref{lem::half_green_function_distant}, the latter term will be the main contribution to the asymptotics of $G_{u_n}(x,y(n))$.
Finally, by \eqref{eq:uniform_bound_green},
\begin{align*}
S_{3}&=\mathbb{E}_{u_n}\left[\widetilde{G}_{u_n}(x+S_{\tau_x},y(n)),m_n<\tau_x<\infty\right]\\
&\leq C\mathbb{P}_{u_n}\left[m_n<\tau_x<\infty\right]\\
&\leq C\left(\mathbb{P}_{u_n}\left(\vert S({m_n})-m_n\mu^{u(n)}\vert\leq c m_n,m_n<\tau_x<\infty\right)+\mathbb{P}_{u_n}\left(\vert S({m_n})-m_n\mu\vert> c m_n\right)\right),
\end{align*}
for $c$ small enough so that $\inf_{u\in V(u_0)} d(\mu^{u(n)},\partial\mathbb{R}_{>0}^2):=d>c$. On the one hand, it follows from Lemma~\ref{lem:large_deviation_uniform} that there exists $\eta'$ such that 
\begin{equation*}
    \mathbb{P}_{u_n}(\vert S({m_n})-m_n\mu^{u(n)}\vert> c m_n)\leq \exp(-\eta'm_n)
\end{equation*}
for all $m_n\geq 1$ and $u\in V(u_0)$. On the other hand,
\begin{align*}
\mathbb{P}_{u_n}(\vert S({m_n})-m_n\mu^{u(n)}\vert\leq c m_n,m_n<\tau_x<\infty)\leq& \sup_{y\in B(x+S({m_n}),c_{m_n})}\mathbb{P}_{u_n}(\tau_{y}<\infty)\\
\leq&  \sup_{y\in B(x+S({m_n}),c_{m_n})}\mathbb{P}_{u_n}(\tau^i_{y}<\infty),
\end{align*}
for $i\in\{1,2\}$, where $\tau^i_{y}=\inf\{n\geq 0: x_i+S_i(n)\leq 0\}$. Since $S$ is skip-free and there exists $\mu_i>0$ such $\mu_i^u\geq \mu_i$ for all $u\in V(u_0)$, by \eqref{eq:survival_probability_small_mu_1} and \eqref{eq:survival_probability_small_mu_2} there exists $c_{u_0}>0$ independent of $u\in V(u_0)$ such that $\mathbb{P}_{u_n}(\tau^i_{y}<\infty)\leq c_{u_0}^{y}$. Hence,
by the choice of $c$ we have 
\begin{equation}
\label{eq:asymptotic_survival_away_drift}
\sup_{y\in B(x+S({m_n}),c_{m_n})}\mathbb{P}_{u_n}(\tau_{y}<\infty)\leq \exp(-c_{u_0}(d-c)m_n).
\end{equation}
Finally, there exists $c>0$ independent of $u\in V(u_0)$ such that $S_3\leq \exp(-cm_n)$, and as long as $\log(t_n)=o(m_n)$, 
\begin{equation}\label{eq:green_S3}
S_3=o(\sqrt{t_n}).
\end{equation}
Putting \eqref{eq:green_S_1}, \eqref{eq:green_S2} and \eqref{eq:green_S3} together yields  
\begin{equation*}
    \sqrt{t_n}G_{u_n}(x,y(n)){}\xrightarrow[t\rightarrow\infty]{}A(u_0)\mathbb{P}_{u_0}(\tau_x=\infty),
\end{equation*}
and setting $B(u_0)= A(u_0)\vert \mu^{u_0} \vert$ yields the desired limit for $G_{u_n}$, since $t_n=\frac{\vert y(n)\vert}{\vert \mu^{u_n}\vert}\sim_{n\rightarrow \infty}\frac{\vert y(n)\vert}{\vert \mu^{u_0}\vert}$. Using \eqref{eq:relation_Green_function_central_change} yields then the final result.

The proof of the second statement is similar. The unique change is to use Proposition~\ref{prop:asymptotic_Green_horizontal_final} instead of Theorem \ref{thm:asymptotic_green_Ney_spiter}, and the renewal relation 
\begin{equation*}
    G(x,y)=\widehat{G}_i(x,y)-\mathbb{E}(\widehat{G}(x+S(\tau^{\bar{\imath}}_x),y),\tau^{\bar{\imath}}_x<\tau^{i}_x)
\end{equation*}
in \eqref{eq:renewal_decomposition}. The main term $S_{21}$ is then asymptotically equivalent to 
\begin{equation*}
    \frac{A(e_i)V(y_2(n))f(y_2'/\sqrt{\vert y(n)\vert})}{\vert y(n)\vert^{3/2}}\mathbb{E}_{e_i}(x_i+S_{i}(m_n),\tau_x>m_n).
\end{equation*}
Then,  using that $x_{\bar{\imath}}$ is harmonic for $S$ killed at $\tau^{i}$ under $\mathbb{P}_{e_i}$ (because $S$ is skip-free),
\begin{align*}
&\mathbb{E}_{e_i}(x_{\bar{\imath}}+S_{\bar{\imath}}(m_n),\tau_x>m_n)\\
&=\mathbb{E}_{e_i}(x_{\bar{\imath}}+S_{\bar{\imath}}(m_n),\tau_x^{i}>m_n)-\mathbb{E}_{e_i}(x_{\bar{\imath}}+S_{\bar{\imath}}(m_n),\tau_x^{i}>m_n\geq \tau^{\bar{\imath}})\\
&=x_{\bar{\imath}}-\mathbb{E}_{e_i}\left(\mathbb{E}_{e_i}\left(x_{\bar{\imath}}+S_{\bar{\imath}}(\tau^{\bar{\imath}})+\widetilde{S}_{\bar{\imath}}(m_n-\tau^{\bar{\imath}}),\tau_{x+S(\tau^{\bar{\imath}}_x)}^{i}>m_n-\tau^{\bar{\imath}}_x\right),\tau^{\bar{\imath}}_x \leq m_n,\tau^{\bar{\imath}}<\tau^i_x)\right),
\end{align*}
where $\widetilde{S}$ has the same law of $S$ and is independent of $\tau^{\bar{\imath}}$. By harmonicity of $x_{\bar{\imath}}$ for $S$ killed at $\tau^i$, we get 
\begin{equation*}
    \mathbb{E}_{e_i}\left(x_{\bar{\imath}}+S_{\bar{\imath}}(\tau^{\bar{\imath}})_x+\widetilde{S}_{\bar{\imath}}(m_n-\tau^{\bar{\imath}}_x),\tau_{x_{\bar{\imath}}+S_{\bar{\imath}}(\tau^{\bar{\imath}}_x)}^{i}>m_n-\tau^{\bar{\imath}}_x\right)=x_{\bar{\imath}}+S_{\bar{\imath}}(\tau^{\bar{\imath}}_x),
\end{equation*}
so that 
\begin{equation*}
\mathbb{E}_{e_i}(x_{\bar{\imath}}+S_{\bar{\imath}}(m_n),\tau_x>m_n)=x_{\bar{\imath}}-\mathbb{E}_{e_i}\left(x_{\bar{\imath}}+S_{\bar{\imath}}(\tau^{\bar{\imath}}),\tau^{\bar{\imath}} \leq m_n, \tau^{\bar{\imath}}_x<\tau^i_x\right).
\end{equation*}
Since $\mathbb{P}(m_n<\tau_x^{\bar{\imath}}<\infty )=o(\exp(-t_n^{\epsilon}))$ for some $\epsilon>0$ as in \eqref{eq:asymptotic_survival_away_drift}, we have 
\begin{align*}
\mathbb{E}_{e_i}(x_{\bar{\imath}}+S_{\bar{\imath}}(m_n),\tau_x>m_n)\sim x_{\bar{\imath}}-\mathbb{E}_{e_i}\left(x_{\bar{\imath}}+S_{\bar{\imath}}(\tau_x^{\bar{\imath}}), \tau^{\bar{\imath}}_x<\tau^i_x\right)&=x_{\bar{\imath}}-\mathbb{E}_{e_i}\left(x_{\bar{\imath}}+S_{\bar{\imath}}(\tau_x)\right)\\
&=-\mathbb{E}_{e_i}\left(S_{\bar{\imath}}(\tau_x)\right),
\end{align*}
where we used that $x_{\bar{\imath}}+S_{\bar{\imath}}(\tau^i_x)=0$ because $S$ is skip-free.
\end{proof}

\begin{proof}[Proof of Corollary~\ref{cor:MB}]
By Theorem \ref{thm:asymptotic_inside_green}, if $\{y(n)\}_{n\geq 0}$ is such that $\vert y(n)\vert\rightarrow \infty$ with $\frac{y(n)}{\vert y(n)\vert}\rightarrow u\in \Sigma$, then if $u\not\in\{e_1,e_2\}$, for $x\in \mathbb{Z}_{>0}^2$ we have
\begin{align}
K(x,y(n))\rightarrow e^{\langle \phi(u), x-x_0\rangle}\frac{\mathbb{P}_{u_n}(\tau_x=\infty)}{\mathbb{P}_{u_n}(\tau_{x_0}=\infty)}&=C(u)e^{\langle\phi(u),x\rangle}\mathbb{P}_{u_n}(\tau_x=\infty)\nonumber\\
&=C(u)\left(e^{\langle\phi(u), x\rangle}-\mathbb{E}\left(e^{\langle\phi(u), x+S(\tau)\rangle},\tau<\infty\right)\right)\label{eq:expression_h_inside}\\&:=\bar{h}_{\phi(u)}(x),\nonumber
\end{align}
for some $C(u)>0$ and if $u=e_i$, $i\in\{1,2\}$,
\begin{align}
K(x,y(n))&\rightarrow e^{\langle\phi(e_i), x-x_0\rangle}\frac{x_i-\mathbb{E}_{e_i}(x_{i}+S_{i}(\tau_x))}{{x_0}_i-\mathbb{E}_{e_i}({x_0}_{i}+S_{i}(\tau_{x_0}))}\nonumber\\
&=C(e_i)\left(x_ie^{\langle \phi(e_i), x\rangle}-\mathbb{E}\left(e^{\langle\phi(e_i), x+S(\tau)\rangle}(x_{i}+S_{i}(\tau_x))\right)\right):=\bar{h}_{\phi(e_i)}(x),\label{eq:expression_h_axis}
\end{align}
for some $C(e_i)>0$. Remark that as $d(x,\partial \mathbb{Z}_{>0})\rightarrow \infty$, $\mathbb{P}_{u_n}(\tau_x=\infty)\rightarrow 1$ and $\mathbb{E}_{e_i}(x_{\bar{\imath}}+S_{\bar{\imath}}(\tau^i_x),\tau^{\bar{\imath}}_x>\tau^i_x)\rightarrow \mathbb{E}_{e_i}(x_{\bar{\imath}}+S_{\bar{\imath}}(\tau^i_x))=x_i$, so that if $u\not\in\{e_1,e_2\}$,
\begin{equation*}
    \bar{h}_{\phi(u)}(x)\sim_{d(x,\partial \mathbb{Z}_{>0})\rightarrow \infty} e^{\langle\phi(u),x\rangle}
\end{equation*}
and if $u=e_i$, $i\in\{1,2\}$,
\begin{equation*}
   \bar{h}_{\phi(e_i)}(x)\sim_{d(x,\partial \mathbb{Z}_{>0})\rightarrow \infty} x_{\bar{\imath}}e^{\langle\phi(e_i), x\rangle}.
\end{equation*}   
Since $u\mapsto \phi(u)$ is injective, we deduce that $u\mapsto \bar{h}_{\phi(u)}$ is injective. Hence, $\Phi:u\mapsto  \bar{h}_{\phi(u)}$ is a bijection from $\Sigma$ to $\partial_M^{\mathbb{Z}_{>0}^2}S$. Remark that $\Phi$ is continuous: this is clear from \eqref{eq:expression_h_inside} on $\Sigma\setminus\{e_1,e_2\}$, and \eqref{eq:expression_h_axis} for $e_1$ is obtained as a limit of \eqref{eq:expression_h_inside} as $u$ goes to $e_1$ (the derivative appearing because of the term $C(u)$). Since $\Sigma$ is compact and $\Phi$ is a continuous bijection, $\Phi$ is a homeomorphism from $\Sigma$ to $\partial_M^{\mathbb{Z}_{>0}^2}S$.

It remains to prove that $\frac{h_{\phi(u)}}{h_{\phi(u)}(1,1)}=\bar{h}_{\phi(u)}$. Let us first suppose that $u\in \Sigma\setminus\{e_1,e_2\}$. The function $h_{\phi(u)}$ is harmonic in $\mathbb{Z}_{>0}^2$, so that $h_{\phi(u)}=\int_{\Sigma}\bar{h}_{\phi(w)}dm_u(w)$ for some positive measure $m_u$ on $\Sigma$. We have  $h_{\phi(u)}(x)\sim e^{\langle\phi(u),x\rangle}$ as $d(x,\partial \mathbb{Z}_{>0})\rightarrow \infty$. Hence, since $\bar{h}_{\phi(w)}(x)\sim e^{\langle\phi(w),x\rangle}$ (resp.\ $\bar{h}_{\phi(e_i)}(x)\sim x_{\bar{\imath}}e^{\langle\phi(e_i), x\rangle}$ as $d(x,\partial \mathbb{Z}_{>0})\rightarrow \infty$ for $w\in \Sigma\setminus \{e_1,e_2\}$ (resp.\ $w=e_i$, $i\in\{1,2\}$) we must have $m_u=c\delta_{u}$ for come constant $c$. Since $\frac{h_{\phi(u)}(1,1)}{h_{\phi(u)}(1,1)}=\bar{h}_{\phi(u)}(1,1)=1$, $c=1$ and thus 
\begin{equation*}
    \frac{h_{\phi(u)}}{h_{\phi(u)}(1,1)}=\bar{h}_{\phi(u)}.
\end{equation*}
Then for $i\in\{1,2\}$, by continuity of $\Phi$ and Proposition \ref{prop:boundary_expression_hf},
\begin{equation*}
    \frac{h_{\phi(e_i)}}{h_{\phi(e_i)}(1,1)}=\lim_{u\rightarrow e_i}\frac{h_{\phi(u)}}{h_{\phi(u)}(1,1)}=\lim_{u\rightarrow e_i}\bar{h}_{\phi(u)}=\bar{h}_{\phi(e_i)}.
\end{equation*}
Since $\phi$ is a homeomorphism from $\Sigma$ to $\overline{\mathcal{G}_0}$ with $\phi^{-1}(a_0,b_0)=\frac{\nabla K(a_0,b_0)}{\Vert \nabla K(a_0,b_0)\Vert}$, we get the statement of the corollary.
\end{proof}

\appendix
\addcontentsline{toc}{section}{Appendices}

\section{Proof of Proposition \ref{prop:equiv_proba_micro}} \label{Appendix:llt}

The proof of Proposition \ref{prop:equiv_proba_micro} is very similar to the one of Theorem 5 in \cite{DeWa-15} which provides a local limit theorem for zero-drift random walks in cones (the cone being a half-space for us) : the two differences are the presence of a small drift in our case and the fact that we are dealing with walks with small negative jumps in the vertical direction, which greatly simplifies the proof. The first step is to get good estimates on the heat kernel of the Brownian motion with drift inside $\mathbb{R}\times \mathbb{R}_{>0}$.
\begin{lem}
\label{lem:asymptotic_brownian_drift}
Let $(B_t)_{t\geq 0}$ be a Brownian motion with drift $\mu$ and covariance $\Sigma$. Denote by $\tau^{B}_x$ the exit time from $\mathbb{R}\times \mathbb{R}_{\geq 0}$ for $x+B_t$. Then, writing $C=\Sigma^{-1/2}$ and $K^B(x,y)dy=d\mathbb{P}\left(x+B_t=y, \tau^B_x>t\right)$,
 we have for $x,y\in \mathbb{R}\times \mathbb{R}_{\geq 0}$
\begin{equation*}
    K^B_t(x,y)=\left(1-e^{-\frac{2x_2y_2}{t \Sigma_{11}}}\right)\frac{\exp\left(-\frac{\Vert C(y-x-t\mu)\Vert^2}{2t}\right)}{2\pi t\det(\Sigma)}.
\end{equation*}
In particular, for all $x,y\in \mathbb{R}\times \mathbb{R}_{\geq 0}$
\begin{equation*}
    K^B_t(x,y)\leq \frac{2x_2y_2}{t \Sigma_{11}}\frac{\exp\left(-\frac{\Vert C(y-x-t\mu)\Vert^2}{2t}\right)}{2\pi t\det(\Sigma)},
\end{equation*}
and as $\frac{x_2y_2}{\Sigma_{11}}\leq \theta_t t$ with $\theta_t$ going to zero,
\begin{equation*}
    K^B_t(x,y)\sim \frac{2x_2y_2}{t \Sigma_{11}}\frac{\exp\left(-\frac{\Vert C(y-x-t\mu)\Vert^2}{2t}\right)}{2\pi t\det(\Sigma)},
\end{equation*}
with $\sim$ only depending on $\theta_t$. 
\end{lem}
\begin{proof}
This lemma is a consequence of Girsanov theorem and the reflection formula. For $v\in \mathbb{R}^2$, denote by $s_v$ the orthogonal symmetry with respect to a one-dimensional vector space generated by $v$, so that $s_v(x)=x-2\frac{\langle x,v^{\perp}\rangle}{\langle v^{\perp},v^{\perp}\rangle}v^{\perp}$ for any $v^{\perp}\not=0$ with $\langle v^{\perp},v\rangle=0$. Recall that if $W$ is a standard two-dimensional Brownian motion with zero drift and $\tau_x^v=\inf\{t\geq 0: x+W_t\in \mathbb{R}v\}$, the reflection formula yields
\begin{align*}
d&\mathbb{P}(x+W_t=y, \tau_x^v>t)\\
=&\frac{1}{2\pi t}\left(\exp\left(-\frac{\Vert y-x\Vert^2}{2t}\right)-\exp\left(-\frac{\Vert s_v(y)-x\Vert^2}{2t}\right)\right)\\
=&\frac{1}{2\pi t}\left(\exp\left(-\frac{\Vert y-x\Vert^2}{2t}\right)-\exp\left(-\frac{\Vert y-x-2\frac{\langle y,v^{\perp}\rangle}{\langle v^{\perp},v^{\perp}\rangle}v^{\perp}\Vert^2}{2t}\right)\right)\\
=&\frac{1}{2\pi t}\left(\exp\left(-\frac{\Vert y-x\Vert^2}{2t}\right)-\exp\left(-\frac{\Vert y-x\Vert^2-4\frac{\langle y-x,v^{\perp}\rangle\cdot\langle y,v^{\perp}\rangle}{\langle v^{\perp},v^{\perp}\rangle}+4\frac{\langle y,v^{\perp}\rangle^2}{\langle v^{\perp},v^{\perp}\rangle}}{2t}\right)\right)\\
=&\frac{1-e^{-\frac{2\langle y,v^{\perp}\rangle\langle x,v^{\perp}\rangle}{t\langle v^{\perp},v^{\perp}\rangle}}}{2\pi t}\exp\left(-\frac{\Vert y-x\Vert^2}{2t}\right).
\end{align*}
Write $\widehat{B}_t=CB_t$ and $\widetilde{B}_t=\widehat{B}_t-tC\mu$, so that $\widehat{B}$ is a Brownian motion with covariance identity and drift $C\mu$ and $\widetilde{B}$ is a standard Brownian motion. Using the change of variable $u\mapsto Cu$ and Girsanov theorem applied to $\widetilde{B}$, we get
\begin{align*}
K_t^B(x,y)&=d\mathbb{P}(x+B_t=y, \tau^B_x>t)\\
&=(\det C)d\mathbb{P}(Cx+\widehat{B}_t=Cy, \tau_{x}^{Ce_2}>t)\\
&=(\det C)e^{\langle C\mu, C(x-y)\rangle-\frac{t\Vert C\mu\Vert^2}{2}}d\mathbb{P}(Cx+\widetilde{B}_t=Cy, \tau_{x}^{Ce_1}>t).
\end{align*}
Then, since $\langle \Sigma^{1/2}e_2,Ce_1\rangle=0$, we can apply the previous computation to get
\begin{align*}
K_t^B(x,y)=&(\det C)e^{C\mu\cdot (x-y)-\frac{t\Vert C\mu\Vert^2}{2}}\frac{1-e^{-\frac{2\langle Cy,C^{-1}e_2\rangle\langle x,C^{-1}e_2\rangle}{te_1^T\Sigma e_1}}}{2\pi t}\exp\left(-\frac{\Vert Cy-Cx\Vert^2}{2t}\right)\\
=&\frac{1-e^{-\frac{2y_2x_2}{t\Sigma_{11}}}}{2\pi t\sqrt{\det \Sigma}}\exp\left(-\frac{\Vert Cy-Cx\Vert^2-2t\langle C(y-x),C\mu\rangle+t^2\langle C\mu,C\mu\rangle}{2t}\right)\\
=&\frac{1-e^{-\frac{2y_2x_2}{t\Sigma_{11}}}}{2\pi t\sqrt{\det \Sigma}}\exp\left(-\frac{\Vert C(y-x-t\mu)\Vert^2}{2t}\right).
\end{align*}
The rest of the lemma is deduced by using the formula $1-e^{-u}\leq u$ for $u\in \mathbb{R}$ and $1-e^{-u}\sim u$ as $u$ goes to $0$.
\end{proof}
We next turn to the proof of Proposition \ref{prop:equiv_proba_micro}.
\begin{proof}[Proof of Proposition \ref{prop:equiv_proba_micro}]
The proof of this proposition is exactly the same as the one of \cite[Thm~5]{DeWa-15} and \cite[Prop.~2.4]{DuRaTaWa-22}. We introduce the main steps without proving them in details, since there are exactly the same as in the aforementioned references.

\subsubsection*{Conditioned central limit theorem starting away from the boundary} Since $S(n)$ has all moments uniformly bounded with respect to $u$, by \cite{Ein-89}, for all $r\geq 3,\gamma>0$ there exists $\gamma>0$ small enough and a coupling between $S_n-n\mu$ and a Brownian motion $B^u$ with variance $\Sigma^u$ such that 
\begin{equation*}
    \mathbb{P}(\sup_{0\leq t\leq n}\vert S_{\lfloor t\rfloor}- t\mu-B_{t}^u\vert >n^{1/2-\gamma})\leq Cn^{-r},
\end{equation*}
with $C$ independent of $u\in \Sigma$. As in \cite[Lem.~20]{DeWa-15} we deduce from this by choosing $\gamma=1/4+2\epsilon$ and using Lemma \ref{lem:asymptotic_brownian_drift} that for any $A\subset \mathbb{R}^2$ compact with non-empty interior, uniformly on $x\in \mathbb{H}_1$ with $x_2\geq n^{1/4-\epsilon}$ with $\epsilon$ small (only depending on $\gamma, r$) ,
\begin{align}
\mathbb{P}_u(x+S(n)\in n\mu+\sqrt{n}A,\tau^1_x>n)&\sim \int_{n\mu+\sqrt{n}A\cap \mathbb{R}\times \mathbb{R}_{\geq 0}}K^{B^u}_n(x,y)dy\nonumber\\
&\sim \frac{c(u)x_2\mu_2}{2\pi\Sigma^u_{11}\sqrt{\det (\Sigma^u)}}\int_{A}\exp\left(-\frac{\Vert C^u y\Vert^2}{2}\right)dy,\label{eq:equivalent_proba_away}
\end{align}
(remark that we use the fact that $n\mu+\sqrt{n}A\in \mathbb{R}\times \mathbb{R}_{\geq 0}$ for $n$ large enough, due to the condition  $n^{-1/2}=o(u_2)$)
and 
\begin{equation}
\label{eq:bound_proba_away}
\mathbb{P}_u(\vert S(n)\vert\geq t,\tau^1_x>n)\leq Cx_2\mu_2\exp\left(-c t^2\right)dy+O(n^{-r}),
\end{equation}
for some $C,c>0$ and $r$ large enough.

\subsubsection*{Conditioned central limit theorem starting close the boundary} Introduce the stopping times $\nu_n^x=\inf(m\geq 1, x_2+S_2(m)\geq n^{1/4-\epsilon})$. Then, as in \cite[Lemma 14]{DeWa-15} $\mathbb{P}(\nu_n\geq n^{1/2-\epsilon})=O(e^{-Cn^{\epsilon}})$ and by Lemma \ref{lem:large_deviation_uniform}, 
\begin{equation*}
    \mathbb{P}(\nu_n\leq n^{1/2-\epsilon}, \vert S(\nu_n)-\nu_n\mu^u\vert\leq n^{1/4-\epsilon/4})\geq 1-\exp(-Cn^{-\epsilon'})
\end{equation*}
for some $\epsilon'>0$. Hence, we have by \eqref{eq:equivalent_proba_away}
\begin{align*}
&\mathbb{P}_u(x+S(n)\in n\mu+\sqrt{n}A,\tau^1_x>n)\\
=&\mathbb{E}\Bigg(\mathbb{P}\bigg(x+S(\nu_n^x)\in x+S(\nu_n^x)+(n-\nu_n^x)\mu+\sqrt{n}(A+n^{-1/2}(\nu_n^x\mu-S(\nu_n^x)),\\
&\hspace{9cm} \tau^1_{x+S(\nu_n^x)}>n-\nu_n^x\bigg), \nu_n^x<\tau^1_x\Bigg)\\
=&\mathbb{E}\Bigg(\mathbb{P}\bigg(x+S(\nu_n^x)\in x+S(\nu_n^x)+(n-\nu_n^x)\mu+\sqrt{n}(A+o(n^{-1/2-\frac{\epsilon}{4}})), \tau^1_{x+S(\nu_n^x)}>n-\nu_n^x\bigg),\\
&\hspace{7cm} \nu_n<n^{1/2-\epsilon}, \nu_n^x<\tau^1_x\Bigg)+o\left(\exp\left(-Cn^{-\epsilon'}\right)\right)\\
&\sim c(u)\mathbb{E}\left(x_2+S_2(\nu_n^x),\nu_n<n^{1-\epsilon}, \nu_n^x<\tau^1_x\right)\frac{\mu_2}{2\pi n\sqrt{\Sigma^u_{11}\det(\Sigma^u)}}\int_{A}\exp\left(-\frac{\Vert C^u y\Vert^2}{2}\right)dy,
\end{align*}
uniformly on $u\in \Sigma$ and $x\in \mathbb{H}_1, x=o(\sqrt{n})$. Then, using that $\mathbb{P}(\vert x_2+S_2(n)\vert,\nu_n\geq n^{1/2-\epsilon})=O(e^{-Cn^{\epsilon'}})$ as in \cite[Lem.~16]{DeWa-15} and the fact that $S_2(k)-k\mu_2^u$ is a martingale, $\nu_n^x\wedge \tau^1_x$ is a stopping time and $x_2+S_2(\tau^1_x)=0$ (because $S$ is skip-free), we get 
\begin{align*}
\mathbb{E}\left(x_2+S_2(\nu_n^x),\nu_n<n^{1/2-\epsilon}, \nu_n^x<\tau^1_x\right)&=\mathbb{E}\left(x_2+S_2(\nu_n^x\wedge n^{1/2-\epsilon}), \nu_n^x<\tau^1_x\right)+O(e^{-Cn^{\epsilon'}})\\
&=\mathbb{E}\left(x_2+S_2(\nu_n^x\wedge \tau^1_x\wedge n^{1/2-\epsilon})\right)+O(e^{-Cn^{\epsilon'}})\\
&=x_2+\mu_2^u\mathbb{E}(\nu_n^x\wedge \tau^1_x\wedge n^{1/2-\epsilon})+O(e^{-Cn^{\epsilon}}).
\end{align*}
Since $\mu_2^u\leq n^{-1/2+\eta}$, for $\eta$ small enough we have 
\begin{equation*}
    \mu_2^u\mathbb{E}(\nu_n^x\wedge \tau^1_x\wedge n^{1/2-\epsilon})\leq n^{-1/2+\eta}n^{1/2-\epsilon}=o(1),
\end{equation*}
uniformly on $u\in \Sigma$ in the range of the Lemma. Finally, uniformly on such $u$ and $x\in \mathbb{H}_1$, $\vert x\vert=o(\sqrt{n})$,
\begin{equation*}
    \mathbb{P}_u(x+S(n)\in n\mu+\sqrt{n}A,\tau^1_x>n)\sim_{n\rightarrow \infty}\frac{c(u)x_2\mu_2}{2\pi n\Sigma^u_{11}\sqrt{\det(\Sigma^u)}}\int_{A}\exp\left(-\frac{\Vert C^u y\Vert^2}{2}\right)dy.
\end{equation*}
Similarly, we get for $x\in \mathbb{H}_1$ and $t\geq 0$,
\begin{equation*}
    \mathbb{P}_u(\vert S(n)-n\mu^u\vert\geq \sqrt{n}t,\tau^1_x>n)\leq Cx_2\mu_2\exp\left(-c^u t^2\right)dy+o(n^{-r}).
\end{equation*}

\subsubsection*{Local limit theorem starting close the boundary}
 Using the same proof as the one of \cite[Thm~5]{DeWa-15} with the local limit theorem from Spitzer \cite[Ch.~7 P10]{Spi-64} which is valid for any random walk whose steps set generates $\mathbb{Z}^2$, we deduce
\begin{equation*}
    \mathbb{P}_u(x+S(n)=y,\tau^1_x>n)=\frac{c(u)x_2\mu_2}{\Sigma^u_{11}2\pi n\sqrt{\det(\Sigma^u)}}\exp\left(-\frac{\Vert C^u (y-n\mu^u)\Vert^2}{2n}\right)+o\left(\frac{\mu^u_2x_2}{n}\right)
\end{equation*}
with $o(\cdot)$ uniform on $u$ and $x,y\in \mathbb{H}_1$ with $\vert x\vert\leq \sqrt{n}$, and
\begin{equation*}
    \mathbb{P}_u(x+S(n)=y,\tau^1_x>n)\leq \frac{Cx_2\mu_2}{n}\exp\left(-c^u \Vert y-n\mu^u\Vert^2\right)dy+o\left(n^{-(1+r)}\right).\qedhere
\end{equation*}
\end{proof}

\section*{Acknowledgments}
Part of this work has been written during a stay at the University
of M\"unster of the first and second authors, who would like to thank the Institute for Mathematical Stochastics for the very good working conditions. We thank the editors for their kind invitation to submit a paper to this special issue in honour of J.~W.~Cohen. We further thank Ivo Adan, Onno Boxma, Johan van Leeuwaarden and Michel Mandjes for continuous and fruitful interactions since 2010. We thank Ivo Adan for useful discussions at an early stage of the project. KR thanks Alin Bostan and Lucia Di Vizio for interesting discussions on singular random walks.

\end{document}